\documentclass[12pt]{amsart}

\usepackage{amsmath,graphicx}
\usepackage{adjustbox}
\usepackage{amssymb}
\usepackage{tikz, tikz-cd}
\usepackage{fullpage}
\usepackage[all]{xy}
\usepackage[margin=1in]{geometry}
\usepackage{amsthm}
\usepackage{amsfonts}
\usepackage{bbm}
\usepackage{comment}
\usepackage{amsmath}
\usepackage{amsthm}
\usepackage{bbm}
\usepackage{array} 
\usepackage{color}
\usepackage[utf8]{inputenc}
\usepackage[english]{babel}
\usepackage{hyperref}
\newcolumntype{L}{>{$}l<{$}}
\DeclareMathSymbol{\shortminus}{\mathbin}{AMSa}{"39}

\newtheorem{theorem}{Theorem}[section]
\newtheorem{lemma}[theorem]{Lemma}

\newtheorem{corollary}[theorem]{Corollary}
\newtheorem{conjecture}[theorem]{Conjecture}
\newtheorem{proposition}[theorem]{Proposition}
\theoremstyle{definition}  
\newtheorem{definition} [theorem] {Definition}

\newtheorem{remark} [theorem] {Remark}
\newtheorem{question} [theorem] {Question}

\theoremstyle{definition}

\newcommand{\F}{{\mathbb{F}}}

\newcommand{\Q}{{\mathbb{Q}}}

\newcommand{\R}{{\mathbb{R}}}
\newcommand{\Z}{{\mathbb{Z}}}

\newcommand{\Spin}{\text{Spin}}

\newcommand{\rot}{\text{rot}}
\newcommand{\tb}{\text{tb}}
\newcommand{\tbb}{\overline{\mathrm{tb}}}
\newcommand{\CF}{\widehat{CF}}

\newcommand{\CFKINFTY}{{CFK}^{\infty}}
\newcommand{\HFKINFTY}{{HFK}^{\infty}}
\newcommand{\HF}{HF}
\newcommand{\HFhat}{\widehat{HF}}

\newcommand{\s}{\mathfrak{s}}
\newcommand{\rk}{\text{rk}}

\newcommand{\Gh}{\widehat{\Gamma}}
\newcommand{\Mrel}{\mathcal{M}_{\mathrm{rel}}}
\newcommand{\Drel}{\Delta_{\mathrm{rel}}}
\newcommand{\Tpun}{T^{\bullet}}
\newcommand{\CFDhat}{\widehat{CFD}}
\newcommand{\CFAhat}{\widehat{CFA}}
\DeclareMathOperator{\Th}{th}

\usepackage{comment}
\usepackage[abs]{overpic}
\usepackage{graphicx}
\newcommand{\din}{d}

\title{Contact cosmetic surgery on Legendrian knots in  integer homology sphere $L$-spaces }
\author{Apratim Chakraborty} 
\address[]{IAI, TCG CREST, Kolkata, India}

\email{apratimn@gmail.com}

\author{Swarup Kumar Das}
\address[]{IAI, TCG CREST, Kolkata, India}
\email{swarupdas.math@gmail.com}

\author{Tanushree Shah}
	\address{Department of Mathematics \\ Chennai Mathematical Institute \\ India}
	\email{tanushreebshah@gmail.com}

\begin{document}
\begin{abstract}

We extend the study of contact cosmetic surgeries to Legendrian knots in integer homology sphere $L$-spaces . We prove that the contact cosmetic surgery conjecture holds for all non-trivial Legendrian knots in this setting, with the possible exception of Lagrangian slice knots. Our argument adapts and refines techniques from the $S^3$ case to the broader context of $L$-spaces, incorporating constraints arising from Heegaard Floer theory.

\end{abstract}
\maketitle

\section{Introduction}

In this paper, we establish that a contact analogue of the cosmetic surgery conjecture holds for Legendrian knots in integer homology sphere $L$-spaces, with the possible exception of a small and geometrically constrained family of knots.

We begin by recalling the smooth cosmetic surgery conjecture in a general setting. Let $K$ be a knot in a closed oriented $3$–manifold $Y$. Two surgeries $Y_K(r)$ and $Y_K(r')$ are called \textit{cosmetic} if they are diffeomorphic, and \textit{truly cosmetic} if they are orientation–preserving diffeomorphic. In the case $Y=S^3$, it is known that while cosmetic surgeries exist, the only known truly cosmetic surgeries occur on the unknot. The cosmetic surgery conjecture predicts that non-trivial knots admit no truly cosmetic surgeries. This conjecture has been extensively studied in $S^3$, particularly using Heegaard Floer homology and correction terms, and these results play an essential role in our arguments.

We now turn to the contact setting. Let $(Y,\xi)$ be a tight contact structure on an $L$-space integer homology sphere, and let $L \subset (Y,\xi)$ be a Legendrian knot. Contact $(r)$–surgery on $L$ is obtained by removing a standard neighborhood of $L$ and gluing back a solid torus whose meridian has slope $r+\tb(L)$, extending the contact structure over the torus as any tight contact structure. As in the case of $S^3$, when $r \neq 1/n$, there are generally multiple possible contact surgeries corresponding to a single smooth surgery coefficient.

Two contact surgeries on $L$ with distinct smooth coefficients are said to be \textit{cosmetic} if the resulting contact manifolds are contactomorphic. Since contactomorphisms are orientation-preserving, we do not distinguish between cosmetic and truly cosmetic surgeries in the contact category.

This leads to the following natural generalization of the cosmetic surgery conjecture to the contact setting.

\begin{conjecture}[Contact cosmetic surgery conjecture]
Let $(Y,\xi)$ be an integer homology sphere. Any Legendrian knot in $(Y,\xi)$ that is not smoothly an unknot admits no cosmetic contact surgeries.
\end{conjecture}

This conjecture is a natural extension of the smooth cosmetic surgery conjecture and is closely related to understanding the extent to which contact surgery descriptions of contact manifolds are unique.
As a first step toward this problem, it is useful to isolate a class of knots for which the surgery behavior is more tractable and closely tied to the classical $S^3$ setting. This leads to the notion of local knots, where the surgery problem effectively reduces to one in $S^3$.
\subsection*{Local knots}

A knot $K \subset Y$ is said to be \emph{local} if it is contained
in an embedded $3$–ball in $Y$. Equivalently, $K$ determines
a knot in $S^3$. In this situation, performing smooth $r$–surgery
along $K$ affects only the $3$–ball containing the knot,
and the resulting manifold is diffeomorphic to
$Y \# S^3_K(r)$,
where $S^3_K(r)$ denotes the manifold obtained by $r$–surgery
on the corresponding knot in $S^3$. 

\begin{remark}
In the case of local knots, the limiting conditions
remain unchanged from the $S^3$ setting,
since the ambient manifold differs only by a
connected sum.
Thus, the contact manifolds obtained are of the form $Y \# S^3_K(r)$,
and the analysis reduces to that in
\cite{ES}.
\end{remark}

We now establish an obstruction to performing contact surgery on more general classes of knots in integer homology sphere $L$-spaces. The $L$-space condition imposes strong Floer-theoretic constraints, and these constraints allow us to adapt and refine arguments from the $S^3$ case. Our main result establishes this conjecture for almost all Legendrian knots.
\begin{remark}
    
Throughout this paper, $Y$ will denote an $L$-space integer homology $3$–sphere (ZHS$^3$). We fix once and for all a positively co-oriented tight contact structure $\xi$ with nonvanishing Heegaard Floer contact invariant on $Y$ and work in the contact manifold $(Y,\xi)$. 

%Unless otherwise stated, $Y$  will have contact surgery diagrams which consist solely of contact $(-1)$–surgeries. In particular, every contact manifold under consideration is obtained via Legendrian surgery and is therefore Stein fillable.

One example to keep in mind is the Poincar\'e homology sphere $-P^3$ equipped with a positively co-oriented tight contact structure arising as the boundary of its Stein filling.
\end{remark}

\begin{theorem}\label{contactcosmeticthm}
Let $(Y,\xi)$ be a contact integer homology sphere $L$-space with non-trivial Heegaard Floer contact invariant. The contact cosmetic surgery conjecture holds for all non-trivial Legendrian knots in $(Y,\xi)$, except possibly for $\pm 2$ surgery on a Legendrian knot $L$ whose smooth knot type $K$ satisfies
$\tau_Y(K)=0, \quad g(K)=2, \quad \overline{\tb}(K)=-1$. 
\end{theorem}

Here $\tau_Y(K)$ denotes the Heegaard Floer smooth $\tau$–invariant of the knot in $Y$. The proof proceeds by constructing explicit contact $(\pm1)$–surgery diagrams representing candidate cosmetic surgeries and comparing their $d_3$–invariants. The $L$-space hypothesis allows us to control correction terms and eliminate potential coincidences except in the narrow family described above.

As in the $S^3$ setting, when $r \neq 1/n$ there may be multiple contact $(r)$–surgeries on a fixed Legendrian knot. This leads to two refinements of cosmetic behavior. We say two contact $(r)$–surgeries are \textit{weakly cosmetic} if the resulting contact manifolds are contactomorphic, and \textit{strongly cosmetic} if the resulting contact structures are isotopic. Since the smooth surgery coefficient is fixed, the underlying smooth manifold is canonically identified in this situation.

Rather than formulating a precise conjecture in this direction, we ask:

\begin{question}
Which Legendrian knots in tight contact integer homology sphere $L$-spaces  admit weakly-cosmetic or strongly-cosmetic surgeries?
\end{question}
\begin{remark}
    We further analyze contact surgeries on Legendrian unknots in $L$-space homology spheres and observe a striking dichotomy: certain Legendrian unknots admit contact surgeries with no cosmetic counterparts, while in other cases a given contact surgery is contactomorphic to infinitely many distinct surgeries on the same knot.
\end{remark}
We establish several structural restrictions analogous to the $S^3$ case. In particular, Floer-theoretic constraints coming from the $L$-space condition severely restrict the possibility of strongly-cosmetic surgeries for non-trivial knot types, especially for $L$-space knots realizing the maximal Thurston–Bennequin bound $\tb(L)=2\tau_Y(L)-1$. Ni–Wu showed that knots in $S^3$ admitting truly cosmetic surgeries must have vanishing $\tau$-invariant \cite{NiWu}. We extend Ni–Wu’s result to integer homology sphere $L$-spaces .

\begin{theorem}
\label{tauresult}
Let $Y$ be an  $L$-space  integer homology sphere and $K\subset Y$ a knot.
If $K$ admits truly cosmetic surgeries in $Y$—that is, there exist distinct slopes
$r_1\neq r_2$ with $Y_{r_1}(K)\cong Y_{r_2}(K)$ as oriented $3$-manifolds—then the
Ozsváth–Szabó knot Floer invariant $\tau_Y(K)=0$. 
\end{theorem}

When $Y$ is an integer homology $L$-space equipped with a contact structure $\xi$  whose contact invariant is nontrivial, the smooth $\tau$-invariant $\tau_Y(K)$ agrees with the contact $\tau$-invariant $\tau_\xi(K)$ defined by Hedden \cite{HEDDENcontacttau}. Using the lower bound on $\tau_\xi(K)$, we then obtain the desired constraint on the maximal Thurston-Bennequin number.

Hanselman \cite{Han} obtained further restrictions on knots in $S^3$ admitting truly cosmetic surgery using immersed curve techniques. We extend Hanselman's result to knots in integer homology $L$-spaces. This  gives additional restrictions which can be summarized in the following theorem.

\begin{theorem}\label{adaptedHanselman}
Let $p, q, q' > 0$ with $\gcd(p,q)=\gcd(p,q')=1$. If $Y_{p/q}(K) \cong Y_{-p/q'}(K)$, then:
\begin{enumerate}

\item[\rm (a)] $q = q'$, so the slopes are $\{p/q,\, -p/q\}$;
\item[\rm (b)] $q^2 \equiv -1 \pmod{p}$.
\item[\rm (c)] The slopes either $\pm 2$ or $\pm \frac{1}{ q}$
for some $q \in \Z_{> 0}$
\qquad \text{in the second case};
\item[\rm (d)] if the slopes are $\pm 2$, then $g(K) = 2$.

\end{enumerate}
\end{theorem}

Finally, as in the smooth category, it is necessary to treat Legendrian unknots separately. While smooth surgeries on an unknot in an integer homology sphere are highly constrained, contact surgeries exhibit a mixture of rigidity and flexibility. In certain cases, a Legendrian unknot admits contact surgeries that are unique up to contactomorphism, while in many other cases, a given contact surgery is contactomorphic to infinitely many distinct contact surgeries on the same unknot. This phenomenon may be viewed as a contact analogue of Rolfsen twists, though the $L$-space condition introduces additional constraints through the classification of tight contact structures on solid tori and lens spaces.

The proofs combine convex surface theory, detailed analysis of contact surgery diagrams, properties of the Farey graph, and Floer-theoretic obstructions arising from correction terms and contact invariants.
\begin{remark}
    We extend our proof to integer homology $L$-spaces by considering simplified contact surgery diagrams of the knot that we obtain after doing handle-slides to make $d_3$ computations feasible.
\end{remark}

\begin{comment}
To begin with we would like to study property P for anull homologous knots in general 3-manifold.
A knot K in S³ has Property P if every 3-manifold obtained by performing non-trivial Dehn surgery along K is not $S^3$. The conjecture states that all knots, except the unknot, have Property P. 
\textcolor{blue}{we need to generalise property P}\\
A proposed generalisation of property P conjecture: For a compact connected oriented 3-manifold $M$ no non-trivial surgery on a non-trivial knot gives a manifold homeomorphic to $M$.\ 
idea of proof: use filtered chain complex of HF. Only need to check for 1/q surgery (homology constraints).
Rational surgery formula. Rank inequality

Property P and R are true for null homologous knots in QHS3. In our case for all knots.

We propose the cosmetic property P and R conjecture for Legendrian knots in $M,\xi_{tight}$
There are multiple generalizations of cosmetic surgery question 
\begin{conjecture}
    
\end{conjecture}
QHS^3 null homologous knots. $c\neq 0$

\end{comment}
	\subsection*{Acknowledgements}
We are grateful to John Etnyre for numerous helpful discussions.	
\section{Preliminaries}

\subsection{Contact Surgery}

We briefly recall the construction of contact surgery,
standard references include \cite{DingGeiges04,Geiges08,Honda00}.

Let $(Y,\xi)$ be a co-oriented contact $3$–manifold and
let $L\subset (Y,\xi)$ be a Legendrian knot.
Denote by $\tb(L)$ and $\rot(L)$ the
Thurston–Bennequin and rotation invariants of $L$.

For $r\in\mathbb{Q}$, contact $r$–surgery along $L$
is defined as follows.
Let $\nu(L)$ be a standard tubular neighborhood of $L$
with convex boundary and dividing slope equal to the contact framing.
Remove $\nu(L)$ and reglue a solid torus so that the meridian
is identified with a curve of slope $r$ measured relative to
the contact framing.
The resulting manifold is denoted $Y_L(r)$ and carries a
naturally induced contact structure $\xi_L(r)$.

Topologically, contact $r$–surgery corresponds to smooth
Dehn surgery with coefficient
$r+\tb(L)$
with respect to the Seifert framing
(see \cite{DingGeiges04}).

Of particular importance are the cases $r=\pm1$.
Contact $(-1)$–surgery coincides with Weinstein
$2$–handle attachment and therefore produces a Stein
cobordism \cite{Gompf98}.
Contact $(+1)$–surgery does not in general admit a Stein
interpretation but is well-defined as a contact operation
\cite{DingGeiges04}.

More generally, for $r\in\mathbb{Q}$,
contact $r$–surgery may be realized as a sequence of
contact $(\pm1)$–surgeries along Legendrian push–offs of $L$,
determined by a continued fraction expansion of $r$
\cite{DingGeiges04}.
This description allows one to compute invariants,
such as the $d_3$–invariant, directly from the
associated surgery diagram.
\subsection{$d_3$–Invariant}

Let $(Y,\xi)$ be a closed, co-oriented contact $3$–manifold
such that $c_1(\xi)$ is torsion. The $d_3$–invariant is a
homotopy invariant of the oriented $2$–plane field underlying $\xi$
(see \cite{Gompf98}).

Let $(X,J)$ be a compact almost complex $4$–manifold with
$\partial X = Y$ such that the complex tangencies induced by $J$
coincide with $\xi$ along the boundary. Following
\cite{Gompf98}, we use the normalized formula
\[
d_3(\xi)
=
\frac{1}{4}
\left(
c_1^2(X,J)
-
3\sigma(X)
-
2(\chi(X)-1)
\right),
\]
where $\sigma(X)$ and $\chi(X)$ denote the signature and Euler
characteristic of $X$, respectively.
The normalization ensures that
$d_3(\xi_{\mathrm{std}})=0$ for the standard contact structure
on $S^3$.

Suppose now that $(Y,\xi)$ is obtained from
$(S^3,\xi_{\mathrm{std}})$ by contact surgery along a Legendrian
link $L=L_1\cup \dots \cup L_k$.
Let $X$ be the $4$–manifold obtained by attaching $2$–handles
according to the surgery diagram, and let $Q$ denote the linking
matrix of the corresponding smooth surgery description.
Define
\[
c = (\mathrm{rot}(L_1),\dots,\mathrm{rot}(L_k))^T.
\]
Then, as shown in \cite{Gompf98} (see also \cite{DG}),
\[
d_3(\xi)
=
\frac{1}{4}
\left(
c^T Q^{-1} c
-
3\sigma(Q)
-
2k
\right)
+
q,
\]
where $k$ is the number of components of the surgery link,
$\sigma(Q)$ is the signature of $Q$, and
$q$ denotes the number of contact $(+1)$–surgeries
in the diagram.

In particular, if the diagram consists only of contact
$(-1)$–surgeries (i.e.\ Legendrian surgeries),
then $q=0$ and
\[
d_3(\xi)
=
\frac{1}{4}
\left(
c^T Q^{-1} c
-
3\sigma(Q)
-
2k
\right).
\]

\subsection{Invariants of rational homology spheres}

\subsubsection*{Casson-Walker invariant}
The Casson invariant of an integral homology sphere $Y$ is obtained by studying $SU(2)$-representations of the fundamental group of $Y$. The Casson invariant was extended to rational homology spheres by Walker \cite{Walker}. The Casson-Walker invariant admits a purely combinatorial definition in terms of surgery representations \cite{Lescop}.  If $K \subset Y$ is a knot in an integer homology sphere and $p/q \neq 0$, then the Casson-Walker invariant satisfies the following surgery formula
\begin{equation*}
\lambda(Y_{p/q}(K)) \;=\; \lambda(Y) + \lambda(L(p,q))
+ \frac{q}{2p}\,\Delta_K''(1),
\end{equation*}
where $\Delta_K(t)$ is the symmetrized Alexander polynomial of $K$.

\subsubsection*{Casson-Gordon invariant}

For a rational homology sphere $Z$,
the Casson-Gordon invariant $\sigma^{\mathrm{CG}}(Z)
$ provides another obstruction to bounding the rational homology ball.  When $Z$ is obtained by $p/q$-surgery on a knot $K$ in an integer homology sphere $Y$, we have
\begin{equation}
\sigma^{\mathrm{CG}}(Z) \;=\;
\sigma^{\mathrm{CG}}(L(p,q)) \;-\; \sigma(K,p),
\end{equation}
where $\sigma(K,p)$ is a certain signature invariant of $K$ associated to the $p$-fold cyclic cover \cite{CassonGordon}.
For lens spaces, $\sigma^{\mathrm{CG}}(L(p,q))$ is determined by the Dedekind sum $s(q,p)$, and the equality
$\sigma^{\mathrm{CG}}(L(p,q_1)) = \sigma^{\mathrm{CG}}(L(p,q_2))$ implies
$\lambda(L(p,q_1)) = \lambda(L(p,q_2))$, since both invariants are
expressed in terms of $s(q,p)$.

\subsection{Heegaard Floer preliminaries}\label{hfprelim}

We review the necessary background from Heegaard Floer homology introduced by Ozsv\'ath and Szab\'o \cite{OSzAnnals,OSzProperties}. 

\subsubsection*{Heegaard Floer three-manifold invariants}

Let $Y$ be a closed, oriented rational homology three-sphere.
Ozsv\'ath and Szab\'o  associated to each $\mathfrak{s} \in {\Spin}^c(Y)$ a collection of $\F[U]$-modules
\[
\HF(Y,\mathfrak{s}),\quad
HF^+(Y,\mathfrak{s}),\quad
HF^-(Y,\mathfrak{s}),\quad
HF^\infty(Y,\mathfrak{s}),
\]
fitting into a long exact sequence
\[
\cdots \longrightarrow
\HF^-(Y,\mathfrak{s}) \longrightarrow
\HF^\infty(Y,\mathfrak{s}) \longrightarrow
\HF^+(Y,\mathfrak{s}) \longrightarrow \cdots\,.
\]
Each of these modules can  be given an absolute $\mathbb{Q}$-grading such that multiplication by $U$ decreases the degree by $2$ \cite{OSzAbsgraded}.

The group $HF^+(Y,\mathfrak{s})$ decomposes as
\[
\HF^+(Y,\mathfrak{s}) \;\cong\; \mathcal{T}^+_{d(Y,\mathfrak{s})}
\;\oplus\; \HF_{\mathrm{red}}(Y,\mathfrak{s}),
\]
where
$\mathcal{T}^+ = \mathbb{F}[U,U^{-1}]/U\cdot\mathbb{F}[U]$ is a $\F[U]$ module commonly known as the tower and $\HF_{\mathrm{red}}(Y,\mathfrak{s})= \HF^+(Y,\mathfrak{s})/\HF^\infty(Y,\mathfrak{s})$ is called the reduced Heegaard Floer homology. The \emph{Heegaard Floer correction term} (or \emph{$\din$-invariant}) of $(Y,\mathfrak{s})$ is defined as
\[
\din(Y,\mathfrak{s})
\;=\;
\min\!\bigl\{\, \mathrm{gr}(x) \;\big|\;
x \in \mathrm{Im}\bigl(\pi \colon HF^\infty(Y,\mathfrak{s}) \to HF^+(Y,\mathfrak{s})\bigr),\;
x \neq 0 \,\bigr\},
\]
where $\mathrm{gr}$ is the absolute $\mathbb{Q}$-grading. Hence, the $\din$-invariant is the minimal absolute grading of the elements belonging in a tower. 

The $\din$-invariant defines a homomorphism from rational homology cobordism group to $\Z$. For a rational homology sphere $Y$, the Casson-Walker invariant $\lambda(Y)$ and the Heegaard Floer $\din$-invariants are related by the formula \cite{OSzAbsgraded}:
\begin{equation*}
|H_1(Y;\mathbb{Z})|\,\lambda(Y) \;=\;
\sum_{\mathfrak{s} \in {\Spin}^c(Y)}
\!\left(\chi\!\left(\HF_{\mathrm{red}}(Y,\mathfrak{s})\right)
\;-\; \frac{1}{2}\,\din(Y,\mathfrak{s})\right).
\end{equation*}

A rational homology sphere $Y$ is called a \emph{Heegaard Floer $L$-space} if $\HFhat(Y,\mathfrak{s}) \cong \F$
for every $\mathfrak{s} \in {\Spin}^c(Y)$, or equivalently if
$HF_{\mathrm{red}}(Y,\mathfrak{s}) = 0$ for all $\mathfrak{s}$. When $Y$ is an integer homology sphere $L$-space, the reduced Floer homology
vanishes and there is a unique ${\Spin}^c$ structure, so the relation between the $\din$ invariants and the Casson-Walker invariant reduces to
\begin{equation*}
\lambda(Y) \;=\; -\frac{1}{2}\,\din(Y).
\end{equation*}

\subsubsection*{The knot Floer chain complex}

A knot  $K \subset Y$ gives rise to a $\Z \oplus\Z$ filtration on  the Heegaard Floer complex $CF^\infty(Y,\mathfrak{s})$. The filtered complex is more conveniently written as $C=\CFKINFTY(Y,K)$. The filtration $\mathcal{F}:S \to \Z \oplus \Z$ is a function on set of generators $S$ of the Heegaard Floer complex with the property that, if $\mathcal{F}(\mathbf{x})=(i, j)$, then $\mathcal{F}(U \cdot \mathbf{x})=(i-1, j-1)$ and $\mathcal{F}(\mathbf{y}) \leq \mathcal{F}(\mathbf{x})$ for all $\mathbf{y}$ having nonzero coefficient in $\partial \mathbf{x}$. Given $I,J \subseteq \Z$, $C\{i \in I, j \in J \}$ will denote the subgroup of $C$ generated by elements $\mathbf{x} \in S$  satisfying $\mathcal{F} (\mathbf{x}) 
\in I \times J$.\\

Let $\mathcal{F}_Y(K, m):=C\{i=0, j \leq m\} \subset {\CF}(Y)$. It follows from the property of filtration that it is a subcomplex of $\CF(Y)$. We obtain a sequence of maps

\[
\iota_K^m: \mathcal{F}_Y(K, m) \longrightarrow \widehat{C F}(Y),
\]

which induce isomorphisms in homology for all sufficiently large integers $m$. The smooth $\tau$-invariant \cite{OSztau}, $\tau_Y(K)$ is defined as follows
\[ \tau_Y(K):=\min \{m \in \mathbb{Z} \mid 
\left.\iota_K^m: \mathcal{F}_Y(K, m) \rightarrow \widehat{C F}(Y) \text { induces a nontrivial map in homology }\right\}.\]

Let  $m(K) \subset -Y$  be the mirror image of the knot  $K \subset Y$ which is obtained by reversing the orientation of the ambient manifold $Y$. We have the following identity relating their $\tau$ invariants \begin{equation*}
\tau_Y(K)=-\tau_{-Y}(m(K)) .
\end{equation*}

\subsubsection*{Rational surgery formula}  Now, we will describe the rational surgery formula for knots in an integer homology sphere $L$-space $Y$. Following \cite{OSzRationalSurgery}, for each integer $s$, define
\[
A_s^+ = C\{\, i \ge 0 \text{ or } j \ge s \,\}, \qquad B^+ = C\{\, i \ge 0 \,\}.
\]
Then $A_s^+$ is a natural quotient complex of $C$, and $B^+$ may be viewed as a subquotient complex. There are two canonical chain maps $v_s, h_s \colon A_s^+ \to B^+$. The map $v_s$ is projection onto $C\{\, i \ge 0 \}$. The map $h_s$ is projection onto $C\{\, j \ge s \}$, followed by identification with $C\{\, j \ge 0 \}$, followed by chain homotopy equivalence from $C\{\, j \ge 0 \}$ to $C\{\, i \ge 0 \}$.\\

Given $p/q \in \Q \setminus \{0\}$, consider the direct sums
\[
\mathbb{A}_i^+ \;=\; \bigoplus_{s \in \Z}
\!\left(s,\, A^+_{\lfloor (i+ps)/q \rfloor}\right), \qquad
\mathbb{B}_i^+ \;=\; \bigoplus_{s \in \Z}
\!\left(s,\, B^+\right),
\]
and define $D^+_{i,p/q} \colon \mathbb{A}_i^+ \to \mathbb{B}_i^+$ by
$D^+_{i,p/q}\{(s,a_s)\} = \{(s,b_s)\}$ with
\begin{equation*}
b_s \;=\; v^+_{\lfloor (i+ps)/q \rfloor}(a_s)
\;+\; h^+_{\lfloor (i+p(s-1))/q \rfloor}(a_{s-1}).
\end{equation*}

\begin{theorem}[Ozsv\'ath-Szab\'o
{\cite[Theorem~9.19]{OSzRationalSurgery}}]\label{thm:rationalsurgeryformula}
Let $\mathbb{X}^+_{i,p/q}$ be the mapping cone of $D^+_{i,p/q}$.
There is a $U$-equivariant isomorphism
$H_*(\mathbb{X}^+_{i,p/q}) \cong \HF^+(Y_{p/q}(K),\, i)$.
\end{theorem}

\subsubsection*{Slice Bennequin inequality and Heegaard Floer contact invariant}

For a Legendrian representative $L$ of $K$ in $S^3$ with standard contact structure, Plamenevskaya\cite{Plamenevskaya2004TB} proved the following lower bound on the Heegaard Floer $\tau$–invariant:
\begin{theorem}[\cite{Plamenevskaya2004TB}]\label{taubound}  
 \[
\tb(L)+|\rot(L)| \le 2\tau(K)-1.
\]  
\end{theorem}

Hedden generalized the definition of the $\tau$-invariant for a null-homologous knot $K$ in a rational homology sphere $Y$, obtaining a family of $\tau$-invariants $\{\tau_{\alpha}(Y,K)\}$, where $\alpha \in \widehat{HF}(Y)$. In particular, we will be interested in the contact $\tau$ invariant $\tau_{\xi}(K)$. It is defined as follows:

\begin{definition}
    If $Y$ is a rational homology sphere and $\xi$ is a contact structure, we define the \textbf{contact} $\tau$ invariant $\tau_{\xi}(K):= -\tau_{c(\xi)}(-Y,K)  $ where $c(\xi)$ is the Heegaard Floer contact invariant. 
\end{definition}

He also proved the following generalization of Plamenevskaya's result:

\begin{theorem}[\cite{HEDDENcontacttau}]
Let $(Y,\xi)$ be a contact three-manifold with nontrivial Ozsv\'ath-Szab\'o contact invariant
$c(\xi)\in \widehat{HF}(-Y)$. Then for a null-homologous knot $K \hookrightarrow Y $ anwith a Seifert surface $F$, we have

\[
\tb(\widetilde K)\;+\;\bigl|\rot_F(\widetilde K)\bigr|
\;\le\;
2\,\tau_\xi(K)\;-\;1,
\]
where $\widetilde K$ is a Legendrian representative of $K$. 
\end{theorem}

We note that when $Y$ is an integer homology sphere, $\bigl|\rot_F(\widetilde K)\bigr|$ is independent of the choice of the Seifert surface $F$.  Moreover, when $Y$ is an integer homology sphere $L$-space with a contact structure having nontrivial contact invariant then the contact $\tau$ invariant coincides with the smooth $\tau$ invariant of $K$, $\tau_Y(K)$. Therefore, Hedden's result can be stated as follows:

\begin{corollary}[\cite{HEDDENcontacttau}]\label{corhedden}
Let $(Y,\xi)$ be a  contact  integer homology sphere $L$-space with nontrivial Ozsv\'ath-Szab\'o contact invariant. Then for any knot $K \hookrightarrow Y $ and a Legendrian representative  $\widetilde K$ of $K$  we have

\[
\tb(\widetilde K)\;+\;\bigl|\rot(\widetilde K)|
\;\le\;
2\,\tau_Y(K)\;-\;1.
\] 
\end{corollary}

\subsection{Bordered Heegaard Floer homology and the immersed curve invariant}\label{subsec:bordered}
 
We summarize the bordered Heegaard Floer machinery of
Lipshitz-Ozsv\'ath-Thurston \cite{LOT} and its reformulation by
Hanselman-Rasmussen-Watson \cite{HRW1,HRW2} via immersed curves in
the punctured torus, restricting attention to the case of torus
boundary, which is the case we will need.
 
\subsubsection*{Bordered invariants for torus boundary}
 
Let $T = \R^2 / \Z^2$ denote the standard torus. The torus algebra
$\mathcal{A} = \mathcal{A}(T)$ is a unital differential graded algebra
over $\F$ generated by certain Reeb chords on $T$, with two distinguished
idempotents $\iota_0, \iota_1$ summing to the identity. Lipshitz, Ozsv\'ath,
and Thurston associate to a compact oriented 3-manifold $M$ with a
parametrization $\phi : T \to \partial M$ two homotopy-equivalence
classes of curved bordered modules
\[
\CFDhat(M, \phi), \qquad \CFAhat(M, \phi),
\]
called respectively the \emph{type-D} and \emph{type-A} invariants;
these are modules and $\mathcal{A}_\infty$-modules over $\mathcal{A}$, respectively.
\subsubsection*{The Hanselman-Rasmussen-Watson immersed curve invariant}
 
Hanselman, Rasmussen, and Watson \cite{HRW1, HRW2} reinterpret the type-D
module of a 3-manifold $M$ with torus boundary as a geometric object: a
finite collection of decorated immersed curves in the punctured torus
\[
T^{\bullet} \;=\; (T \setminus \{z\})
\]
where $z \in T$ is a fixed marked point. Specifically, they construct,
from $\CFDhat(M, \phi)$, a finite collection  of
immersed curves
\[
\Gh(M, \phi) \;=\; \gamma_0 \cup \gamma_1 \cup \dots \cup \gamma_k \;\subset\; \Tpun,
\]
up to regular homotopy with each component decorated with a local system; the components are connected to one another by a finite
collection of \emph{grading arrows} encoding the relative Maslov gradings.

\subsubsection*{The immersed curve knot invariant $\Gh(K)$}

Now let $Y$ be an integer homology sphere $L$-space, and let
$K \subset Y$ be a knot. Since $H_1(Y;\Z)=0$, the knot $K$ has a preferred
Seifert longitude. Let
\[
M_K \;=\; Y \setminus \nu(K)
\]
be the knot exterior, and choose the standard parametrization
\[
\phi_K : T \longrightarrow \partial M_K
\]
which identifies the standard meridian and longitude on $T$ with the meridian
$\mu_K$ and the Seifert longitude $\lambda_K$ of $K$. We define the immersed
curve invariant of $K \subset Y$ to be the immersed curve invariant of this
bordered knot exterior:
\[
\Gh(K) \;:=\; \Gh(M_K,\phi_K).
\]
Thus $\Gh(K)$ is a finite collection of decorated immersed curves in the
punctured torus associated to the bordered manifold $Y \setminus \nu(K)$.

As in the case of knots in $S^3$, it is useful to pass to the infinite cyclic
cover of the punctured torus in which the longitude direction closes up and the
meridian direction does not. We identify this cover with the punctured cylinder
\[
(\R/\Z)\times \R \setminus \{(0,s+\tfrac12)\mid s\in \Z\},
\]
and we use the vertical coordinate as the Alexander-height coordinate. In this
cylindrical model, the invariant is a finite collection of oriented immersed
curves
\[
\Gh(K) \;=\; \gamma_0 \cup \gamma_1 \cup \cdots \cup \gamma_k,
\]
together with possible local systems and grading arrows encoding relative
Maslov gradings. We suppress the local systems from the notation, but they are
part of the invariant.

The meridional filling of $M_K$ is $Y$. Therefore the bordered-Floer pairing
theorem gives
\[
\rk \widehat{HF}(Y)
\;=\;
\min \bigl|\Gh(K)\pitchfork \mu\bigr|,
\]
where $\mu$ denotes a meridian line in the punctured torus. Since $Y$ is an
integer homology sphere $L$-space, we have
\[
\rk \widehat{HF}(Y)=1.
\]
Consequently, after putting $\Gh(K)$ in minimal position with $\mu$, there is a
unique distinguished component $\gamma_0$ of $\Gh(K)$ meeting $\mu$. This
component wraps once around the punctured cylinder homologically. We orient
$\gamma_0$ from left to right. All other components of $\Gh(K)$ are
null-homologous in the punctured cylinder and can be placed in a neighborhood
of the vertical line through the marked points.

\subsubsection*{Relation with the $UV=0$ knot Floer complex}

The curve system $\Gh(K)$ is equivalent to the $UV=0$ quotient of the knot
Floer complex of the pair $(Y,K)$. let $CFK(Y,K)$ denote the
knot Floer complex over $\F[U,V]$. Define
\[
C_{UV=0}(Y,K)
\;:=\;
CFK(Y,K)\otimes_{\F[U,V]} \F[U,V]/(UV).
\]
This quotient remembers the vertical and horizontal differentials of the knot
Floer complex and kills every term involving both a positive power of $U$ and a
positive power of $V$.

The quotient $C_{UV=0}(Y,K)$ can be recovered directly from $\Gh(K)$. Let
$\mu$ denote a meridional line in the punctured cylinder, slightly perturbed
away from the marked points. The generators of the corresponding immersed-curve
complex are the intersection points
\[
x \in \Gh(K)\pitchfork \mu.
\]
The differential counts immersed bigons whose boundary lies on $\Gh(K)$ and
$\mu$. After replacing each marked point by a nearby pair of basepoints
$w$ and $z$, a bigon $\phi$ contributes the monomial
$
U^{n_w(\phi)}V^{n_z(\phi)}$. Since the coefficient ring is $\F[U,V]/(UV)$, only bigons covering $w$-basepoints
or $z$-basepoints, but not both, contribute. Thus the immersed-curve complex
obtained from $\Gh(K)$ is chain homotopy equivalent to
$C_{UV=0}(Y,K)$ (See Theorem 51 \cite{HRW2} and Theorem 1.2 \cite{hanselman2023knotfloerhomologyimmersed}).

Conversely, starting from a reduced model of $C_{UV=0}(Y,K)$, one constructs
$\Gh(K)$ by translating the vertical and horizontal arrows in the complex into
arcs in the punctured cylinder.

The full knot Floer complex may contain mixed
terms $U^aV^b$ with $a,b>0$; these terms are not visible in $\Gh(K)$. Hanselman \cite{OSzKnotInvariants} has developed the theory of decorated immersed curves to capture the full knot Floer complex. For our purpose, it suffices to restrict our attention to the weaker invariant $\Gh(K)$.

\subsubsection*{The height grading and the genus}

The vertical coordinate in the cylindrical model records the Alexander grading.
For an intersection point
\[
x \in \Gh(K)\pitchfork \mu,
\]
we say that $x$ has height $s$ if its vertical coordinate lies between the
marked points at heights $s-\frac12$ and $s+\frac12$. Equivalently, the height
of $x$ is the Alexander grading of the corresponding generator of the
$UV=0$ knot Floer complex.

The Seifert genus of $K \subset Y$ is the maximal Alexander
grading in knot Floer homology \cite{Ozsvth2003HolomorphicDA}:
\[
g(K)
\;=\;
\max\bigl\{|A| \mid \widehat{HFK}(Y,K,A)\neq 0\bigr\}.
\]
In the immersed-curve picture, this is the largest Alexander height reached by
the reduced curve system. Equivalently,
\[
g(K)
\;=\;
\max\bigl\{|s| \mid \Gh(K)\pitchfork \mu
\text{ has a generator of height } s\bigr\}.
\]
\subsubsection*{The vertical-segment counts $n_s$ and $n$}

We now put $\Gh(K)$ in the pulled-tight position used by Hanselman. In this position, it is vertical everywhere except near a small neighborhood of marked points where it is round. These small round neighborhoods of marked points are referred to as pegs. The
distinguished component $\gamma_0$ contains one segment which leaves a
neighborhood of the meridian line and wraps once around the cylinder. The
remaining relevant segments are vertical, up to a small perturbation, and run
between consecutive pegs.

A vertical segment is said to have height $s$ if it connects the peg at height
$s-\frac12$ to the peg at height $s+\frac12$. Define
\[
n_s
\;:=\;
\#\{\text{vertical segments of } \Gh(K) \text{ at height } s\},
\]
and define
\[
n
\;:=\;
\sum_{s\in \Z} n_s.
\]
Only finitely many of the integers $n_s$ are nonzero. There is a $\pi$-rotation symmetry of $\Gh(K)$ which implies
\[
n_s \;=\; n_{-s}.
\]
 \subsubsection*{Maslov gradings and grading arrows}

The grading arrows on $\Gh(K)$ encode the relative Maslov gradings between the
components of the curve system. Since $Y$ is an integer homology sphere, the
knot Floer complex of $(Y,K)$ has an absolute Maslov grading, normalized so that
the unique generator of $\widehat{HF}(Y)$ lies in grading $d(Y)$.

The Alexander grading of an intersection point is its height, while its Maslov
grading is determined by the rotation, winding, and grading-arrow data on
$\Gh(K)$.

\subsection{Past results in smooth cosmetic surgery}

\subsubsection*{Boyer-Lines: the Casson-Walker obstruction}

The earliest non-trivial obstruction to a cosmetic surgery pair was proved by Boyer-Lines \cite{BoyerLines}, using the Casson-Walker invariant and Walker's surgery formula.
 
\begin{theorem}[Boyer-Lines, \cite{BoyerLines}]\label{thm:BL}
Let $K \subset S^3$ be a knot admitting a truly cosmetic surgery pair.
Then $\Delta_K''(1) = 0$.
\end{theorem}

\subsubsection*{The Heegaard Floer obstructions: Wu and Ni-Wu}
The Heegaard Floer invariants provided sharper constraints to cosmetic surgery.
Wu \cite{Wu} showed that same-sign cosmetic surgeries are obstructed
in any $L$-space integer homology sphere:
 
\begin{theorem}[{\cite{Wu}}]\label{thm:Wu}
Let $K$ be a nontrivial knot in an integer homology sphere $L$-space   $Y$,
and $r, r'$ two distinct rational numbers of the same sign. Then
$Y_r(K) \not\cong Y_{r'}(K)$.
\end{theorem}
 
Together with the Boyer-Lines vanishing of $\Delta_K''(1)$, this is
improved by Ni-Wu , which further constrains the possible slopes:
 
\begin{theorem}[{\cite[Theorem~1.2]{NiWu}, \cite{Wu}}]\label{thm:NiWu}
Let $K$ be a knot in an integer homology sphere $Y$ admitting a truly
cosmetic surgery pair. After possibly replacing $K$ by its mirror
$m(K) \subset -Y$, the cosmetic surgery slopes can be taken to be
$r_1 = p/q_1$ and $r_2 = -p/q_2$ with $p > 0$ and $q_1, q_2 > 0$
coprime to $p$.
\end{theorem}

\subsubsection*{Hanselman's theorem and its statement for $L$-spaces}
 
Hanselman's contribution, substantially sharpens Ni-Wu constraints by giving the precise slope values:
 
\begin{theorem}[{\cite[Theorem~2]{Han}}]\label{thm:Hanselman}
Let $K \subset S^3$ be a nontrivial knot, and suppose
$S^3_r(K) \cong S^3_{r'}(K)$ for $r \neq r'$. Then
\begin{enumerate}
\item[\rm (i)] The pair $\{r, r'\}$ is either $\{\pm 2\}$ or $\{\pm 1/q\}$
for some $q \in \Z_{> 0}$, with
\[
q \;=\; \frac{n_0 + 2 \sum_{s \geq 1} n_s}{4 \sum_{s \geq 1} s^2 n_s}
\qquad \text{in the second case};
\]
\item[\rm (ii)] if $\{r, r'\} = \{\pm 2\}$, then $g(K) = 2$ and
$n_0 = 2 n_1$;
\item[\rm (iii)] if $\{r, r'\} = \{\pm 1/q\}$, then
\[
q \;\leq\; \frac{\Th(K) + 2 g(K)}{2 g(K) (g(K) - 1)}.
\]
\end{enumerate}
\end{theorem}

Hanselman \cite{Han} remarks that ``$S^3$ can be replaced with any integer homology sphere $L$-space''. We will sketch an outline of the argument in Section~4 to justify the claim.

\section{Cosmetic surgery obstructions on integer homology sphere $L$-spaces }

The goal of this section is to extend Ni-Wu's result to the integer homology sphere $L$-space $Y$. The argument closely follows their proof, with some additional grading bookkeeping for $Y$.

\subsection*{ A $d$-invariant formula for rational surgery in integer homology sphere $L$-spaces  }
We begin by reviewing the key invariants from the knot Floer complex and mapping cone construction of the rational surgery formula following the notational convention of \cite{NiWu,Wu}. Throughout this section, we will assume $Y$ is an integer homology sphere $L$-space. The complexes $A_s^+$, $B^+$ and the mapping cone  $D^+_{i,p/q} \colon \mathbb{A}_i^+ \to \mathbb{B}_i^+$ are defined from knot Floer complex of knot $K$ in $Y$ as in Section \ref{hfprelim}. Let $\mathfrak{A}_k^+ := H_*(A_k^+)$,
$\mathfrak{B}^+ := H_*(B^+)$, and denote the induced maps on homology
by $\mathfrak{v}_k^+, \mathfrak{h}_k^+ \colon \mathfrak{A}_k^+ \to
\mathfrak{B}^+$. Let $\mathfrak{D}^+_{i,p/q}$ be the map induced by $D^+_{i,p/q}$ on homology.

\medskip
When $Y$ is an $L$-space integer homology sphere, $\HF^+(Y) \cong \mathcal{T}^+$ and $\HF_{\mathrm{red}}(Y) = 0$.  Consequently
$\mathfrak{B}^+ \cong \mathcal{T}^+$  for every $k$.
Every $\F[U]$-equivariant endomorphism of the tower $\mathcal{T}^+$ is multiplication by $U^n$ for some $n \geq 0$.  Define the non-negative integers $V_k$ and $H_k$ by
\begin{equation}\label{eq:VkHkdef}
\mathfrak{v}_k^+\big|_{\mathfrak{A}_k^T} \;\sim\; U^{V_k}, \qquad
\mathfrak{h}_k^+\big|_{\mathfrak{A}_k^T} \;\sim\; U^{H_k},
\end{equation}
where $\mathfrak{A}_k^T := U^n \mathfrak{A}_k^+$ for $n \gg 0$
denotes the tower part.  

We state two important lemmas concerning the monotonicity and relation between these invariants.
\begin{lemma}[{\cite[Lemma~2.4]{NiWu}}]\label{lem:monotone}
$V_k \geq V_{k+1}$ and $H_k \leq H_{k+1}$ for all $k$.
Moreover, $V_k = 0$ for $k \geq g(K)$ and $H_k = 0$ for $k \leq -g(K)$.
\end{lemma}
\begin{proof}
    The proof is identical to the case $Y = S^3$ {\cite[Lemma~2.4]{NiWu}}.
\end{proof}

\begin{lemma}[{\cite[Lemma~2.7]{NiWu}}]\label{lem:V0H0}
For any knot $K$ in $Y$, we have $V_0=H_0$.
\end{lemma}
\begin{proof}
    Even in the general case, the conjugation symmetry of doubly pointed Heegaard diagrams $(\Sigma,\alpha,\beta,w,z) \leftrightarrow (-\Sigma,\beta,\alpha, z,w)$ provides a chain homotopy equivalence of $\CFKINFTY$ interchanging the roles of $i$ and $j$ filtrations. The conclusion follows.
\end{proof}

\begin{proposition}\label{prop:UB}

For every $i \in \Z/p\Z$,
\begin{equation}\label{eq:UB}
\din\big(Y_{p/q}(K),\, i\big) \;=\;
\din(Y) \;+\; \din\big(L(p,q),\, i\big)
\;-\; 2\max\!\left\{V_{\left\lfloor i/q \right\rfloor},\;
H_{\left\lfloor (i+p(-1))/q \right\rfloor}\right\}.
\end{equation}

\end{proposition}

\begin{proof}

We adapt the proof of \cite[Proposition~1.6]{NiWu} to the case of an integer homology sphere $L$-space $Y$. Firstly, we establish surjectivity of $\mathfrak{D}^T_{i,p/q}$. Let $\mathfrak{A}_k^T := U^n \mathfrak{A}_k^+$ for $n \gg 0$ and
$\mathfrak{D}^T_{i,p/q}$ the restriction of
$\mathfrak{D}^+_{i,p/q}$ to
$\mathcal{A}_i^T := \bigoplus_s (s, \mathfrak{A}^T_{\lfloor(i+ps)/q\rfloor})$.
This map is surjective: the proof is again identical to
\cite[Lemma~2.8]{NiWu}. We construct a preimage $\xi = \{(s,\xi_s) \}$ of $\eta = \{(s,\eta_s)\}$ by setting
$\xi_{-1} = U^{-H_{\lfloor(i+p(-1))/q\rfloor}} \eta_0$, $\xi_0 = 0$,
and extending to rest of the values of $s$ as described in \cite[Lemma~2.8]{NiWu}.

\medskip
Now, we focus on the tower in $\ker \mathfrak{D}^T_{i,p/q}$. Since $\mathfrak{D}^T_{i,p/q}$ is surjective,
$\mathfrak{D}^+_{i,p/q}$ is also surjective.  By the rational surgery formula
\ref{thm:rationalsurgeryformula}, $\HF^+(Y_{p/q}(K),i) \cong \ker \mathfrak{D}^+_{i,p/q}$.
As in \cite[Lemma~2.9]{NiWu}, $U^n \HF^+(Y_{p/q}(K),i)$ for $n \gg 0$
is identified with a subgroup of $\ker \mathfrak{D}^T_{i,p/q}$.
The $d$-invariant equals the minimal grading in $\ker \mathfrak{D}^T_{i,p/q}$ of the tower generator.
\medskip

Given $\xi \in \mathcal{T}^+$, define
$\rho(\xi) = \{(s,\xi_s)\}_{s \in \Z}$ exactly as in \cite{NiWu}:

\smallskip\noindent

If \[V_{\lfloor i/q \rfloor} \geq H_{\lfloor (i+p(-1))/q \rfloor}\]
Set \[\xi_0 = \xi, \;\
\xi_{-1} = U^{V_{\lfloor i/q\rfloor} - H_{\lfloor(i+p(-1))/q\rfloor}} \xi.\]

\smallskip\noindent
If
\[V_{\lfloor i/q \rfloor} < H_{\lfloor (i+p(-1))/q \rfloor}\]
Set \[\xi_{-1} = \xi, \;\
\xi_0 = U^{H_{\lfloor(i+p(-1))/q\rfloor} - V_{\lfloor i/q \rfloor}} \xi.\]

\smallskip\noindent
\[ \xi_s = 
\begin{cases}  
U^{H_{\lfloor(i+p(s-1))/q\rfloor} - V_{\lfloor(i+ps)/q\rfloor}}\, \xi_{s-1}, & s>0\\

U^{V_{\lfloor(i+p(s+1))/q\rfloor} - H_{\lfloor(i+ps)/q\rfloor}}\, \xi_{s+1}, & s<-1
    
\end{cases}  \]

\smallskip
By \cite[Lemma~2.8]{NiWu}),
$\xi_s = 0$ for $|s| \gg 0$, and it can be directly verified that
$\mathfrak{D}^T_{i,p/q}(\rho(\xi)) = 0$ and $U\rho(\mathbf{1}) = 0$. Hence, the absolute grading of $\rho(\mathbf{1})$ is $\din(Y_{p/q}(K),i)$.

\smallskip

The absolute grading on $\mathbb{B}_i^+$ is chosen so that for the local unknot $O \subset Y$, the grading of $\mathbf{1} \in H_*(\mathbb{X}^+_{i,p/q}(O))
\cong \mathcal{T}^+$ equals $\din(Y) + \din(L(p,q),i)$. When $Y = S^3$, this
reduces to $\din(L(p,q),i)$ as in \cite[Remark~2.3]{NiWu}. This normalization follows from the connected sum formula
$Y_{p/q}(O) \cong Y \# L(p,q)$ and the additivity of $d$-invariants.
The grading of $\rho(\mathbf{1})$ is $\din(Y_{p/q}(O),i)=\din(Y) + \din(L(p,q),i)$ implies the grading of $\mathbf{1} $ in $(0,\mathfrak{B}^+)$ is also $\din(Y) + \din(L(p,q),i)$ as $V_k,H_k=0 \;\ \forall \;\ k\geq 0$ for the local unknot. We consider two possible scenarios.
\begin{enumerate}
    \item Assume $V_{\lfloor i/q \rfloor} \geq H_{\lfloor (i+p(-1))/q \rfloor}$. The component $\xi_0 = \mathbf{1} \in (0, \mathfrak{A}^T_{\lfloor i/q\rfloor})$
maps via $\mathfrak{v}^+_{\lfloor i/q\rfloor}$ to
$U^{V_{\lfloor i/q\rfloor}} \cdot \mathbf{1} \in (0, \mathfrak{B}^+)$.
We obtain the grading
of $\rho(\mathbf{1})$ by shifting the grading of  $\mathbf{1} $ in $\mathbb{B}^+$  down by $2V_{\lfloor i/q\rfloor}$:
\[
\din(Y_{p/q}(K),i) = \din(Y) + \din(L(p,q),i) - 2V_{\lfloor i/q\rfloor}.
\]

\item Assume $V_{\lfloor i/q \rfloor} < H_{\lfloor (i+p(-1))/q \rfloor}$. The component $\xi_{-1} = \mathbf{1}$ maps via
$\mathfrak{h}^+_{\lfloor(i+p(-1))/q\rfloor}$ to
$U^{H_{\lfloor(i+p(-1))/q\rfloor}} \cdot \mathbf{1} \in (0, \mathfrak{B}^+)$,
giving
\[
\din(Y_{p/q}(K),i) = \din(Y) + \din(L(p,q),i) - 2H_{\lfloor(i+p(-1))/q\rfloor}.
\]

\end{enumerate}
Combining both cases yields \eqref{eq:UB}.
 
\end{proof}

\begin{corollary}\label{cor:UB}

\begin{equation}\label{eq:UBineq}
\din\big(Y_{p/q}(K),\, i\big) \;\leq\;
\din(Y) \;+\; \din\big(L(p,q),\, i\big),
\end{equation}

for every $0 \leq i \leq p-1$.  Moreover: Equality holds for all $i$ if and only if $V_k = 0$ for all $k \geq 0$
and $H_k = 0$ for all $k \leq -1$; equivalently, $V_0 = 0$ and $H_0 = 0$.

\end{corollary}

\begin{proof}
    
Since $V_s, H_s \geq 0$, the maximum in formula \eqref{eq:UB} is
non-negative. Hence, we get the inequality \eqref{eq:UBineq}.

As $i$ ranges over $\{0, \ldots, p-1\}$, the index $\lfloor i/q \rfloor$ ranges
over $\{0, 1, \ldots, \lfloor (p-1)/q \rfloor\}$, while
$\lfloor (i - p)/q \rfloor$ ranges over
$\{\lfloor -p/q \rfloor, \ldots, -1\}$.
Therefore, equality for all $i$ is equivalent to
\[
V_k = 0 \;\text{ for } 0 \leq k \leq \lfloor(p-1)/q\rfloor, \qquad
H_k = 0 \;\text{ for } \lfloor -p/q \rfloor \leq k \leq -1.
\]
By Lemma~\ref{lem:monotone}, the first condition is equivalent
to $V_0 = 0$ (which forces $H_0=0$), and the second is
equivalent to $H_{-1} = 0$.
This proves the claim.

\end{proof}
\begin{proposition}\label{prop:V0H0tau}
Let $K$ be a knot in an integer homology $L$-space $Y$. If $V_0 = 0$, then $\tau_Y(K) = 0$.
\end{proposition}

\begin{proof}
First, we show $V_0 = 0$ implies $\tau_Y(K) \leq 0$. Consider the following commutative diagram 
\[
\begin{tikzcd}
H_*(\widehat{A}_0) \ar[r, "(i_A)_*"] \ar[d, "(\hat{v}_0)_*"']
& H_*(A_0^+) \ar[d, "(v_0^+)_*"] \\
\widehat{HF}(Y) \ar[r, hookrightarrow, "(i_B)_*"']
& HF^+(Y)
\end{tikzcd}
\]
Here, $\widehat{A}_0 = C\{ \max \{ i,j \}=0\}$, $i_A$ and $i_B$ are inclusion maps, and $\hat{v}_0$ and $v_0^+$ are projection maps. Since $Y$ is an $L$-space,
$\widehat{HF}(Y) \cong \mathbb{F}$, and $\mathfrak{B}^+ = H_*(B^+) =\mathcal{T}^+ \cong HF^+(Y)$.   

Since $V_0 = 0$, we have $(v_0^+)_*(\mathbf{1}) = \mathbf{1}$
in $\mathcal{T}^+ \cong HF^+(Y)$. There is  a generator $\alpha$ in $H_*(\widehat{A}_0)$ such that  $(i_A)_*(\alpha)= \mathbf{1}$
in $\mathcal{T}^+$. Commutativity gives
\[
(i_B)_* \circ (\hat{v}_0)_*(\alpha)
\;=\; (v_0^+)_* \circ (i_A)_*(\alpha)
\;=\; (v_0^+)_*(\mathbf{1})
\;=\; \mathbf{1} \;\neq\; 0,
\]
so $(\hat{v}_0)_*(\alpha) \neq 0$ since $(i_B)_*$ is injective. Now $\widehat{A}_0 = C\{i = 0,\; j \leq 0\} \cup C\{i < 0,\; j = 0\}$,
and $\hat{v}_0$ projects onto $C\{i = 0\}$, killing the $\{i < 0\}$ part.
So, $(\hat{v}_0)$ is the composition  of $p_0:\widehat{A}_0 \to  C\{i = 0,\; j \leq 0\}$ and $(\iota_K^0)_*:C\{i = 0,\; j \leq 0\} \to C\{i = 0\} $. Therefore,  nontriviality of $(\hat{v}_0)_*$ implies nontriviality of $(\iota_K^0)_*$. Hence,
$\tau_Y(K) \leq 0$.

\smallskip

Now, $\CFKINFTY(-Y,m(K))$ is filtered isomorphic to the dual of $\CFKINFTY(Y,K)$. It follows using duality that $V_0(Y,K)=0$ implies $V_0(-Y,m(K))=0$. Therefore, $\tau_{-Y}(m(K)) \leq 0$ or $\tau_Y(K) \geq 0$. Hence, the conclusion follows.

\end{proof}

\subsection*{Constraints on $d$-invariant  from purely cosmetic surgery}
Recall that, we have the following relation between $\din$ invariants and Casson-Walker invariant-
\[ \left|H_1(Y ; \mathbb{Z})\right| \lambda(Y)=\sum_{\mathfrak{s} \in \operatorname{Spin}^c(Y)}\left(\chi\left(H F_{\mathrm{red}}(Y, \mathfrak{s})\right)-\frac{1}{2} \din (Y, \mathfrak{s})\right).\]

In particular, when $Y$ is an integer homology sphere $L$-space, we have  $\lambda(Y)= -\frac{1}{2} \din(Y)$.\\

We also have \[\lambda(L(p,q))= \sum_{\mathfrak{s} \in \operatorname{Spin}^c(Y)}- \frac{1}{2}  \din (L(p,q), \mathfrak{s}). \]

\begin{proposition}\label{cosmeticconstraint}
    
Let $K$ be a knot in an integer homology $L$-space $Y$ with purely
cosmetic surgery $Z = Y_{p/q_1}(K) \cong Y_{-p/q_2}(K)$,
where $p > 0$ and $q_1, q_2 > 0$.  Then:
\begin{enumerate}
\item $\Delta_K''(1) = 0$.
\item $\displaystyle\sum_{\mathfrak{s} \in \Spin^c(Z)}
\chi(\HF_{\mathrm{red}}(Z,\mathfrak{s})) = \frac{p-1}{2}\,\din(Y)$.
\item $\din(Z,\mathfrak{s}) = \din(Y) + \din(L(p,q_1),\mathfrak{s})$
for each $\mathfrak{s} \in \Spin^c(Z)$.
\end{enumerate}

\end{proposition}

\begin{proof}
   We have two equations for the Casson-Walker and Casson-Gordon invariants of \(Z\):
\[
\begin{aligned}
\lambda(Z)
&= \lambda(Y)+\lambda\bigl(L(p,q_1)\bigr)+\frac{q_1}{2p}\Delta_K''(1) \\
&= \lambda(Y)+\lambda\bigl(L(p,-q_2)\bigr)-\frac{q_2}{2p}\Delta_K''(1), \\
\sigma^{\mathrm{CG}}(Z)
&= \sigma^{\mathrm{CG}}\bigl(L(p,q_1)\bigr)-\sigma(K,p) \\
&= \sigma^{\mathrm{CG}}\bigl(L(p,-q_2)\bigr)-\sigma(K,p).
\end{aligned}
\]

Using the second equation, we obtain
\[
\sigma^{\mathrm{CG}}\bigl(L(p,q_1)\bigr)=\sigma^{\mathrm{CG}}\bigl(L(p,-q_2)\bigr),
\]
which implies
\[
\lambda\bigl(L(p,q_1)\bigr)=\lambda\bigl(L(p,-q_2)\bigr).
\]
Hence, we obtain
\[
\Delta_K''(1)=0.
\]
As a consequence,
\[
\lambda(Z)=\lambda(Y)+\lambda\bigl(L(p,q_1)\bigr)=\lambda(Y)+\lambda\bigl(L(p,-q_2)\bigr).
\]

Now,
\[
\begin{aligned}
\sum_{\mathfrak{s} \in \Spin^c(Z)}
\left(
\chi\bigl(\HF_{\mathrm{red}}(Z,\mathfrak{s})\bigr)
-\frac{1}{2}d(Z,\mathfrak{s})
\right)
= \lambda(Z) 
= -\frac{1}{2}\din(Y)
+ \sum_{\mathfrak{s} \in \Spin^c\!\bigl(L(p,q_1)\bigr)}
\left(-\frac{1}{2}d\bigl(L(p,q_1),\mathfrak{s}\bigr)\right).
\end{aligned}
\]

From Proposition~\ref{prop:UB},
\[
\din\bigl(Z,\mathfrak{s}\bigr)
\le \din(Y)+\din\bigl(L(p,q_1),\mathfrak{s}\bigr).
\]
Therefore, we get
\[
\sum_{\mathfrak{s} \in \Spin^c(Z)}
\chi\bigl(\HF_{\mathrm{red}}(Z,\mathfrak{s})\bigr)
\le \frac{p-1}{2}\,\din(Y).
\]

Now,
\[
\din(Z,\mathfrak{s})
= -\din\bigl(Y_{m(K)}(p/q_2),\mathfrak{s}\bigr)
\ge -\din(Y)-\din\bigl(L(p,q_2),\mathfrak{s}\bigr),
\]
and
\[
\begin{aligned}
\sum_{\mathfrak{s} \in \Spin^c(Z)}
\left(
\chi\bigl(\HF_{\mathrm{red}}(Z,\mathfrak{s})\bigr)
-\frac{1}{2}d(Z,\mathfrak{s})
\right)
= \lambda(Z) 
= -\frac{1}{2}\din(Y)
+ \sum_{\mathfrak{s} \in \Spin^c\!\bigl(L(p,-q_2)\bigr)}
\left(-\frac{1}{2}d\bigl(L(p,-q_2),\mathfrak{s}\bigr)\right).
\end{aligned}
\]

Therefore, we obtain
\[
\sum_{\mathfrak{s} \in \Spin^c(Z)}
\chi\bigl(\HF_{\mathrm{red}}(Z,\mathfrak{s})\bigr)
\ge \frac{p-1}{2}\,\din(Y).
\]

Hence, it follows that
\[
\sum_{\mathfrak{s} \in \Spin^c(Z)}
\chi\bigl(\HF_{\mathrm{red}}(Z,\mathfrak{s})\bigr)
= \frac{p-1}{2}\,\din(Y),
\]
and
\[
\din(Z,\mathfrak{s})
= \din(Y)+\din\bigl(L(p,q_1),\mathfrak{s}\bigr)
\]
for each \(\mathfrak{s} \in \Spin^c(Z)\).

\end{proof}

\begin{proof}[Proof of Theorem \ref{tauresult}] Suppose $K \subset Y$ admits truly cosmetic surgeries with slopes $r_1 \neq r_2$.
By Theorem~\ref{thm:Wu}, the slopes must have opposite signs.
By Theorem~\ref{thm:NiWu}, after possibly replacing $K$ by $m(K) \subset -Y$
(which preserves the hypothesis since $-Y$ is also an integer homology $L$-space),
the surgery is purely cosmetic:
$Z = Y_{p/q_1}(K) \cong Y_{-p/q_2}(K)$ with $p > 0$, $q_1, q_2 > 0$.

By Proposition~\ref{cosmeticconstraint}(3),
\begin{equation}\label{eq:dinZeq}
\din(Z,\mathfrak{s}) = \din(Y) + \din(L(p,q_1),\mathfrak{s})
\end{equation}
for all $\mathfrak{s}$.  That is, equality holds in \eqref{eq:UBineq} for all
$\Spin^c$ structures when applied to the positive surgery $Z = Y_{p/q_1}(K)$.
By Corollary~\ref{cor:UB}, this gives $V_0 = 0$. Hence $\tau_Y(K)= 0$ by Lemma~\ref{lem:V0H0} and Proposition~\ref{prop:V0H0tau}. 

\end{proof}

\begin{corollary}\label{taucosmeticcontraint}
Let $(Y,\xi)$ be a contact  integer homology sphere $L$-space with nontrivial Ozsv\'{a}th-Szab\'{o} contact invariant. If a Legendrian knot  $L$ in $(Y,\xi)$ admits a cosmetic contact surgery, then

\[
\tb(L)\;+\;\bigl|\rot(L)\bigr|
\;\le\;
-\;1.
\] 
\end{corollary}
\begin{proof}
Combined with Corollary~\ref{corhedden}, Theorem~\ref{tauresult}
immediately gives the inequality $\tb(L) + |\rot(L)| \leq -1$.
\end{proof}

\section{Extension of Hanselman's theorem to integer homology sphere $L$-space }\label{sec:Hanselman-extension}

In this section, we will prove Theorem \ref{adaptedHanselman} which extends Hanselman's \cite{Han} theorem to knots in integer homology sphere $L$-spaces. Throughout this section, $Y$ denotes an integer homology sphere $L$-space, and $K \subset Y$ denotes a nontrivial knot. Write
\[
Y_+ \;=\; Y_{p/q}(K), \qquad Y_- \;=\; Y_{-p/q}(K), \qquad p, q > 0,\ \gcd(p,q) = 1.
\]
The invariants $n_s$, $n = \sum_s n_s$, $g(K)$, and $\Th(K)$ are defined from $\Gh(K)$ as in Section~\ref{subsec:bordered}.

For the proof, we first observe that the structural properties of the immersed curves of knots in integer homology sphere $L$-spaces are the same as in the case of $S^3$. Then, using our extension of Ni-Wu's result from the previous section, we see that the remaining results established in \cite{Han} follow by the same argument, except for the $\din$-invariant bookkeeping. We include an outline of the core arguments for completeness.  

\subsection{Structural properties of $\Gh(K)$}
 
When $Y$ is an integer homology sphere $L$-space, the multicurve $\Gh(K)$ still
satisfies the following four structural properties used in Hanselman's proof. 
 
\begin{enumerate}
\item[(I)] $\Gh(K)$ has a unique distinguished component $\gamma_0$
meeting $\mu$, and this component wraps the cylinder $\Tpun$ once
homologically; the remaining components lie in a neighborhood of $\mu$. To see this, recall that the bordered-Floer pairing theorem gives
\[
\rk \HFhat(Y) \;=\; \min\bigl|\Gh(K) \pitchfork \mu\bigr|.
\]
Since $Y$ is an integer homology sphere $L$-space, $\rk\HFhat(Y) = 1$, so exactly one component of $\Gh(K)$ meets $\mu$, which we call $\gamma_0$.
The other components are null-homologous in $\Tpun$.
\item[(II)] Each $\gamma_i$ is unobstructed: no teardrop bounded by
$\gamma_i$ misses a marked point.
\item[(III)] For each $\gamma_i$ with $i \neq 0$, the net winding around
marked points and the net rotation cancel; in particular, any figure-eight
component bounds equal numbers of marked points on each side.
\item[(IV)] $\Gh(K)$ is invariant under the $\pi$-rotation symmetry of
$\Tpun$ about the origin, up to orientation reversal of each component.
\end{enumerate}

\subsection{Surgery formula and the isolated generator}

The key ingredient of the proof is the immersed curve formulation of the Heegaard Floer rational surgery formula.

\begin{theorem}[{\cite{HRW1,HRW2}}]\label{thm:surgery-immersed}
Let $p,q>0$ be coprime and let $\ell^i_{p,q}$ be the line of slope $p/q$ through
$(0,-\tfrac12+i/q+\varepsilon)$.  Then
\[
\HFhat\bigl(Y_{p/q}(K),i\bigr) \cong HF\bigl(\Gh(K),\ell^i_{p,q}\bigr)
\]
as relatively graded $\F$-vector spaces.  If $m$ is the slope of the
nonvertical segment of $\gamma_0$, then
\[
\rk\HFhat\bigl(Y_{p/q}(K)\bigr)=|p-mq|+n|q|.
\]
\end{theorem}

In the previous section, we say that if any knot $K$ in an integer homology sphere $L$-space admits a truly cosmetic surgery then $V_0(K)=H_0(K)=0$. We say $\CFKINFTY(Y,K)$ contains an isolated generator if   $\CFKINFTY(Y,K)$ is chain homotopy equivalent to $\F[U,U^{-1}]\oplus A$
where $A$ is acyclic. The following lemma characterizes this property of knot Floer complex.
\begin{lemma}\label{lem:isolated-gen}
Let $Y$ be an integer homology sphere $L$-space and let $K\subset Y$.
Assume
\[
        V_0(K)=H_0(K)=0.
\]
Then, $\CFKINFTY(Y,K)$ contains an isolated generator. Moreover, the distinguished component $\gamma_0$ of $\Gh(K)$ is horizontal.

\end{lemma}

\begin{proof}
    Hom \cite[Proposition 3.11]{Hom17} proved that the condition $V_0(K)=H_0(K)=0$ implies that $\CFKINFTY(K)$ admits a filtered basis
   with a basis element generating $\HFKINFTY(K)$ and splitting off as a
direct summand; equivalently,
\[
        \CFKINFTY(K)\simeq \CFKINFTY(U)\oplus A.
\] Hence, $\CFKINFTY(K)$ has an isolated generator. It should be noted that Hom's survey restricts its attention to knot Floer complexes for knots in $S^3$. However, it can be easily checked that the proof of Proposition 3.11 in \cite{Hom17} works verbatim for integer homology $L$-spaces as it only uses chain complex level properties which hold in the general case.

Now, using the knot Floer to immersed-curve dictionary, Hanselman observes that an isolated generator property implies that the distinguished component $\gamma_0$ is the horizontal \cite[Section~2.2]{Han}.
\end{proof}

\begin{proposition}\label{prop:gamma0-d-invariant}
Assume that $\gamma_0$ is horizontal.  For each $i\in\Z/p\Z$, let
$x_0^i=\gamma_0\cap \ell^i_{p,q}$.  Then $x_0^i$ represents the tower generator
in $\HFhat(Y_{p/q}(K),i)$ and has Maslov grading
\[
M(x_0^i)=\din\bigl(Y_{p/q}(K),i\bigr).
\]
\end{proposition}

\begin{proof}
By Lemma \ref{lem:isolated-gen}, if $\CFKINFTY(K)$ has an isolated generator tower. 
The horizontal component is the curve representative of the isolated tower
summand in $\CFKINFTY(Y,K)$, exactly as in \cite[Proposition~17]{Han}.  Hence the intersection $x_0^i$ represents the
bottom of the corresponding tower in the mapping-cone description of
$HF^+(Y_{p/q}(K),i)$.  By definition, the absolute grading of this bottom tower
element is $\din(Y_{p/q}(K),i)$.
\end{proof}

\begin{proposition}\label{prop:adapted-NiWu}
Let $p,q,q'>0$ with $\gcd(p,q)=\gcd(p,q')=1$.  If
$Y_{p/q}(K)\cong Y_{-p/q'}(K)$, then:
\begin{enumerate}
\item[\rm (a)] $V_0(K)=H_0(K)=0$.  Consequently $\gamma_0$ is horizontal.
\item[\rm (b)] $q=q'$, so the slopes are $\{p/q,-p/q\}$.
\item[\rm (c)] $q^2\equiv -1\pmod p$.
\end{enumerate}
\end{proposition}
\begin{proof}
Part \textup{(a)} follows from Proposition~\ref{cosmeticconstraint}(1),
Corollary~\ref{cor:UB}, and Lemma~\ref{lem:isolated-gen}.

For \textup{(b)}, Lemma~\ref{lem:isolated-gen} implies that slope $m$ of the distinguished component $\gamma_0$ is equal to zero.  Therefore
Theorem~\ref{thm:surgery-immersed} gives
\[
\rk\HFhat\bigl(Y_{p/q}(K)\bigr)=p+nq,
\qquad
\rk\HFhat\bigl(Y_{-p/q'}(K)\bigr)=p+nq'.
\]
Since the two surgered manifolds are diffeomorphic, the ranks are equal.  The
knot is nontrivial, so $n>0$; hence $q=q'$.

For \textup{(c)}, use the Casson-Walker surgery formula and
$\Delta_K''(1)=0$.  The equality
$Y_{p/q}(K)\cong Y_{-p/q}(K)$ implies
\[
\lambda(Y)+\lambda(L(p,q))
 = \lambda(Y)+\lambda(L(p,-q))
 = \lambda(Y)-\lambda(L(p,q)).
\]
Thus $\lambda(L(p,q))=0$.  Equivalently, the Dedekind sum $s(q,p)$ vanishes;
by the standard vanishing criterion used in \cite[Lemma~23]{Han}, this is
equivalent to $q^2\equiv -1\pmod p$.
\end{proof}

\subsection{Relative gradings}

Assume throughout this subsection that $\gamma_0$ is horizontal.  For an
intersection generator $x\in \HFhat(Y_\pm,\s)$ define
\[
\Mrel(x)=M(x)-\din(Y_\pm,\s).
\]

The diffeomorphism between $Y_{+}$ and $Y_{-}$ defines a map
\[
\phi_i:\Gh(K)\cap\ell^i_{p,q}\longrightarrow \Gh(K)\cap\ell^i_{p,-q}.
\]  Let $\phi:=\coprod_i\phi_i$. It is an isomorphism of ungraded vector spaces.
Now, by the surgery formula  \[ \HF(Y_{\pm},i) \cong HF\bigl(\Gh(K),\ell^i_{p,q}\bigr). \]

The map $\phi_i$ takes an intersection  of the slope line $\ell^i_{p,q}$ with a vertical segment to intersection of a nearby  $\ell^i_{p,q}$ with the same
vertical segment. Since the distinguished generator $x_0^i=\gamma_0\cap \ell^i_{p,q}$ doesn't lie in a vertical segment, it is sent to the corresponding distinguished generator.
For a generator $x$ on a vertical segment, let $A(x)$ be the height of that
segment and let $k(x)$ be the number of marked points, counted with
multiplicity, in the triangle in $\Tpun$ bounded by $\mu$, the relevant arc of
$\ell^i_{p,q}$, and the horizontal arc of $\gamma_0$.

\begin{proposition}[Grading shift]\label{prop:grading-shift}
For every generator $x$ on a vertical segment,
\[
\Drel(x):=\Mrel\bigl(\phi_i(x)\bigr)-\Mrel(x)=1-2|A(x)|-4k(x).
\]
For the distinguished generator $x_0^i$, one has $\Drel(x_0^i)=0$.
\end{proposition}

\begin{proof}
First, observe that \cite[Lemma~13]{Han} holds true in the general case. Then, we see that the rest of the argument which involves  relative grading computation using bigons, holds true in the general case   \cite[Proposition~25]{Han}.  The
$Y$-dependent absolute-grading information doesn't affect the relative grading computation.
\end{proof}

If $Y_+\cong Y_-$, then the graded multisets of relative gradings agree.  Since
$\phi=\coprod_i\phi_i$ is a bijection of the intersection generators, we obtain
\begin{equation}\label{eq:sumDelta}
\sum_{i\in\Z/p\Z}\ \sum_{\substack{x\in \Gh(K)\cap \ell^i_{p,q}\\
x\text{ on a vertical segment}}} \Drel(x)=0.
\end{equation}

\begin{corollary}\label{cor:genus-one}
A nontrivial genus-one knot in $Y$ admits no truly cosmetic surgery.
\end{corollary}

\begin{proof}
If $g(K)=1$, all vertical segments have height $0$.  Hence every vertical
intersection has $A(x)=0$ and $k(x)=0$, so Proposition~\ref{prop:grading-shift}
gives $\Drel(x)=1$.  Since $K$ is nontrivial, at least one vertical segment is
present.  This makes the left side of \eqref{eq:sumDelta} positive, a
contradiction.
\end{proof}

Because $k(x)\ge0$, a generator at height $s$ contributes at most
$1-2|s|$ to $\Drel$.  Summing over all $\mathrm{Spin}^c$ structures in
\eqref{eq:sumDelta} and using $n_{-s}=n_s$, we obtain
Hanselman's rank inequality
\begin{equation}\label{eq:rank-ineq}
n_0\ge \sum_{s\ne0}(2|s|-1)n_s
     =2n_1+6n_2+10n_3+\cdots.
\end{equation}
\subsection{Large slopes: $p/q>1$}

\begin{theorem}\label{thm:large-slopes}
If $Y_+\cong Y_-$ with $p/q>1$, then $p/q=2$, $g(K)=2$, and $n_0=2n_1$.
\end{theorem}

\begin{proof}
First suppose that some $n_s$ with $s\ge2$ is nonzero.  For $p/q>1$, each
slope line meets a length-one vertical segment at most once.  As in
\cite[Theorem~27]{Han}, the set of first $q$ 
${\Spin}^c$ structures $\{ 0,1,\ldots,q-1\}$
 has larger $\rk\HFhat$ than its complement. This
 can be seen by counting intersections  of the height-$s$ segments with  $\ell^i_{p,q}$ and then applying  inequality
\eqref{eq:rank-ineq}. Thus $\phi$ must preserve this set
of $\mathrm{Spin}^c$ structures.
  Summing up and
using the $\din$-invariant identities we obtain
\[
\sum_{i=0}^{q-1}\din(Y_+,i)
= q\din(Y)+\sum_{i=0}^{q-1}\din(L(p,q),i),
\]
\[
\sum_{i=0}^{q-1}\din(Y_-,i)
= q\din(Y)-\sum_{i=0}^{q-1}\din(L(p,q),i).
\]
The two sums must agree, so
$\sum_{i=0}^{q-1}\din(L(p,q),i)=0$.  This contradicts
\cite[Lemma~23]{Han} as we have
$q^2\equiv -1\pmod p$ by Proposition~\ref{prop:adapted-NiWu}.  Thus
$n_s=0$ for all $s\ge2$.

By Corollary~\ref{cor:genus-one}, $g(K)\ne1$.  The genus detection property then forces $g(K)=2$.  For height
$0$ intersections, $\Drel=1$; for height $\pm1$ intersections and $p/q>1$, the
triangles contain no marked points, so $\Drel=-1$.  Equation~\eqref{eq:sumDelta}
therefore gives
\[
q n_0 -2q n_1=0,
\]
and hence $n_0=2n_1$. Now, for determining the slope, height-set analysis
in \cite[Theorem~27]{Han} applies verbatim: if $1<p/q<2$, one obtains a type of
$\mathrm{Spin}^c$ structure whose relative-grading sum is strictly positive;
if $p/q>2$, the necessary balancing type forces the previous unbalanced type to
occur.  Both contradict \eqref{eq:sumDelta}.  Hence $p/q=2$.
\end{proof}

\subsection{Small slopes: $p/q < 1$}

For $p/q<1$, the line $\ell^i_{p,q}$ wraps more than once around the cylinder
between consecutive integer heights.  The triangles in
Proposition~\ref{prop:grading-shift} therefore contain additional marked points,
which strengthens \eqref{eq:rank-ineq}.

\begin{lemma}[Small-slope rank inequality]\label{lem:rank-ineq-strong}
If $Y_+\cong Y_-$ with $p/q<1$, then for each $s\ge1$ there are constants
$a_s>2s^2-1$ such that
\begin{equation}\label{eq:rank-ineq-strong}
n_0\ge \sum_{s\ge1}2n_s a_s
   >2n_1+14n_2+34n_3+\cdots.
\end{equation}
\end{lemma}

\begin{proof}
This is the  averaging argument of \cite[Theorem~28]{Han}.  For a segment
of height $s$, the average number of marked points in the relevant triangles is
strictly larger than $s(s-1)/2$ when $p/q<1$.  Since
$\Drel=1-2s-4k$ for $s>0$, the average negative contribution has size
$a_s=2s+4\overline{k}-1>2s^2-1$.  Combining the positive height-$0$
contribution with the symmetric height $\pm s$ contributions in
\eqref{eq:sumDelta} gives \eqref{eq:rank-ineq-strong}.
\end{proof}

\begin{theorem}\label{thm:small-slopes}
If $K\subset Y$ is nontrivial and $Y_+\cong Y_-$ with $p/q<1$, then $p=1$.
\end{theorem}

\begin{proof}
Assume $p>1$ and write $q=ap+r$ with $0<r<p$.  Each line $\ell^i_{p,q}$ meets
a vertical segment either $a$ or $a+1$ times.  By
\eqref{eq:rank-ineq-strong}, the set of indices for which the height-$0$
segments receive the extra intersection is characterized by strictly larger
$\rk\HFhat$; after cyclic relabelling this set is
$\{0,\ldots,r-1\}$.  Thus $\{0,\ldots,r-1\}$ is preserved by $\phi$. Summing up
and subtracting the negative-surgery sum gives
\[
0=
\sum_{i=0}^{r-1}\bigl(\din(Y_+,i)-\din(Y_-,i)\bigr)
=2\sum_{i=0}^{r-1}\din(L(p,q),i).
\]
Since $r\equiv q\pmod p$ and $q^2\equiv -1\pmod p$, this contradicts
\cite[Lemma~23]{Han}. Therefore $p=1$.
\end{proof}

\subsection{Proof of Theorem~\ref{adaptedHanselman}}

\begin{proof}[Proof of Theorem~\ref{adaptedHanselman}]
Let $Y_r(K)\cong Y_{r'}(K)$ with $r\ne r'$.  By   Proposition~\ref{prop:adapted-NiWu}
then gives $q_1=q_2=q$, $\gamma_0$ horizontal, and
$q^2\equiv -1\pmod p$.

If $p/q>1$, Theorem~\ref{thm:large-slopes} gives $p/q=2$, $g(K)=2$, and
$n_0=2 n_1$, which is the $\{\pm2\}$ case in
Theorem~\ref{thm:Hanselman}.

If $p/q\le 1$, then either $p/q<1$ and Theorem~\ref{thm:small-slopes} gives
$p=1$, or $p/q=1$ and coprimality  gives $p=q=1$.  Thus the slopes are
$\{\pm1/q\}$.  
\end{proof}

\section{Surgery diagram and $d_{3}$ invariant for knots in integer homology sphere}
\normalfont

In this section, we prove that the $d_3$ invariant can be split into summands: the $d_3$ invariant of the starting manifold, which, in our case, is an $L$-space integer homology sphere, and the $d_3$ invariant of the manifold obtained after surgery on the knot. We first prove that the algebraic linking of our knot with the surgery diagram representation of our manifold is zero, and hence the $d_3$ invariant breaks into two summands.
\begin{lemma}
    Let us consider a contact surgery diagram of $(Y(K,r),\xi_r)$, which is obtained from a contact surgery diagram of $(Y,\xi)$ by adding a contact $r$-surgery on a Legendrian knot $L$ inside $Y$. Then, we can perform handle moves to ensure that $L$ has linking number $0$ with all components of the contact surgery diagram of $Y$.
\end{lemma}
\begin{proof}

Let $(Y,\xi)$ be presented by contact surgery on a Legendrian link
\[
\mathcal{L} = L_1 \cup \cdots \cup L_n \subset (S^3,\xi_{\mathrm{std}})
\]
with contact surgery coefficients $\phi = (\phi_1,\dots,\phi_n)$. Let $M_{(\mathcal{L},\phi)}$ denote the associated linking matrix. For any oriented knot $K \subset S^3 \setminus \operatorname{int}(\nu(\mathcal{L}))$, define its linking vector
\[
V_K = (\operatorname{lk}(K,L_1), \dots, \operatorname{lk}(K,L_n))^T \in \mathbb{Z}^n.
\]

Consider the additional Legendrian knot $L \subset Y$, represented in the diagram as a Legendrian knot in $S^3 \setminus \operatorname{int}(\nu(\mathcal{L}))$, and let $V_L$ denote its linking vector with $\mathcal{L}$.

By Kirby calculus arguments, Kim \cite{Kim} proves that there exists a knot $L' \subset S^3 \setminus \operatorname{int}(\nu(\mathcal{L}))$, obtained from $L$ by a sequence of handle slides over the components $L_i$, such that $L$ and $L'$ are isotopic in the surgered manifold $S^3(\mathcal{L},\phi) \cong Y$ and
\[
\operatorname{lk}(L',L_i) = 0 \quad \text{for all } i=1,\dots,n.
\]

 \begin{lemma}[{\cite[Lemma 2.5]{Kim}}]\label{lem:linking_zero}

Let $K \subset S^{3} \setminus \operatorname{int}(\nu(L))$ be an oriented link, and let $X_{K} \in \mathbb{Z}^{n}$ be a solution vector satisfying $M_{(L,\phi)}X_{K} = V_{K}$. Then there exists a link $K^{\prime} \subset S^{3} \setminus \operatorname{int}(\nu(L))$ obtained from $K$ by a sequence of slides over components of $L$ such that $K$ is isotopic to $K^{\prime}$ in the surgered manifold $S^{3}(L,\phi)$ and has a linking vector $V_{K^{\prime}} = 0$. Equivalently, this means $\operatorname{lk}(K^{\prime}, L_{i}) = 0$ for all $i$.
\end{lemma}

We now realize this construction in the contact category. Each smooth handle slide of $L$ over a component $L_i$ can be implemented as a contact handle slide, producing a Legendrian knot whose smooth type agrees with the usual handle slide and such that the resulting contact surgery diagram represents the same contact manifold. Refer to Figure~2 in \cite{DG01} for 2-handle  addition and subtraction. In particular, the effect on linking numbers coincides with the smooth case.

Applying the corresponding sequence of contact handle slides to $L$, we obtain a Legendrian knot $L'$ in the diagram which is isotopic to $L$ in $(Y,\xi)$ and satisfies
\[
\operatorname{lk}(L',L_i) = 0 \quad \text{for all } i.
\]

Replacing $L$ by $L'$ in the contact surgery diagram yields the desired form.

\end{proof}

To see why the $d_3$ invariant breaks into these two summands, recall that the $d_3$ invariant of a contact manifold obtained by Legendrian surgery on a link is determined by the topological invariants of the 4-manifold bounded by the respective surgery presentation. Specifically, Gompf's formula relies on the signature of the linking matrix, the Euler characteristic, and the rotation numbers of the components (which are evaluated using the inverse of the linking matrix). By Lemma \ref{lem:linking_zero}, the Legendrian representative of $K'$ has an algebraic linking number of precisely zero with every component of the background surgery link $L$. Consequently, the linking matrix of the combined link $L \cup K'$ is perfectly block-diagonal. 

Because there are no off-diagonal cross-terms connecting $L$ and $K'$, the algebraic computations are completely separate. The signature of the combined matrix is simply the sum of the signatures of the respective blocks, the Euler characteristic is additive, and the fractional $c^2$ terms (derived from the inverse of the block-diagonal linking matrix) strictly decouple. As a result, the global $d_3$ evaluation separates perfectly into the $d_3$ invariant of the original manifold plus the independent contribution arising from the surgery on $K'$. 

Hence, we only consider the diagrams from the link surgery in $S^3$.

\section{Limiting contact cosmetic surgeries}\label{obstruction}
In this section we prove Theorem~\ref{contactcosmeticthm}. It will be a direct consequence of Propositions~\ref{p1}, \ref{p2}, and~\ref{p3}, which deal with the case of Legendrian knots with $\tb=-1, -2,$ and less than $-2$, respectively, and the fact that Theorem~\ref{taubound} and~\ref{tauresult} show that the contact cosmetic surgery conjecture holds for Legendrian knots with $\tb\geq 0$.

Before proceeding, we make a few remarks.
For smooth $(\pm 2)$–surgeries, both the
contact surgery diagrams and the corresponding
$d_3$–computations are identical to those appearing
in the work of the third author and Etnyre in version 2 of \cite{ES}. In particular, the $d_3$–invariant does not
distinguish all of the resulting contact structures so we give the limiting conditions. For computations we urge the reader to refer to version 2 of \cite{ES} and we give the limiting conditions herein.

For contact $(\pm 1/n)$–surgeries, the $d_3$–invariant computations
and the corresponding surgery diagrams coincide with those given
in Version~1 of \cite{ES} on the arXiv. These arguments, together
with the appendix computations, were omitted from the published
version of that paper after the smooth cosmetic surgery conjecture
for the $(\pm 1/n)$ case was established in
\cite{DEL}. For completeness, we therefore restate
those computations here.

\begin{remark}
For the Legendrian unknot, the conclusions of the present paper
coincide with those obtained in the previous work of third author with Etnyre \cite{ES} in
$(S^3,\xi_{\mathrm{std}})$.
The essential difference is that, in the current setting,
the resulting contact manifolds are of the form
$M \# L(p,q),$
where $M$ denotes the fixed $\mathbb{Z}$–homology sphere under
consideration and $L(p,q)$ is a lens space determined by the
smooth surgery coefficient. Thus, just as in the case of $(S^3,\xi_{\mathrm{std}})$,
Legendrian unknots admit contact cosmetic surgeries.
\end{remark}

	%%%%%%%%%%%%%%%%%%%%%%%%%
	
 Now we prove our theorem.
	
	\begin{proposition}\label{p1}
		The contact cosmetic surgery conjecture holds for Legendrian knots with Thurston-Bennequin invariant $-1$ except possibly for $\pm 2$ surgery on a Legendrian knot $L$ in a prime knot type $K$ with $\tbb(K)=-1$, $\tau(K)=0$.
	\end{proposition}
	\begin{proof}
		Let $L$ be a Legendrian knot with $\tb=-1$. Because we know that $\tau(L)$ must be zero if $L$ admits a cosmetic surgery by Theorem~\ref{tauresult} above and we have the Bennequin type bound in Theorem~\ref{taubound} above, we see that the rotation number of $L$ must be $0$. 
		
		Since $\pm 2$ case is identical to \cite{ES}, we now consider $\pm 1/n$ surgery on $L$. We note that we cannot consider $n=1$ since $-1$ surgery on $L$ would correspond to contact $(0)$ surgery, and this is not well-defined. So we assume $n\geq 2$. Surgery diagrams for these contact surgeries are shown in Figure~\ref{tbm1-general}.
		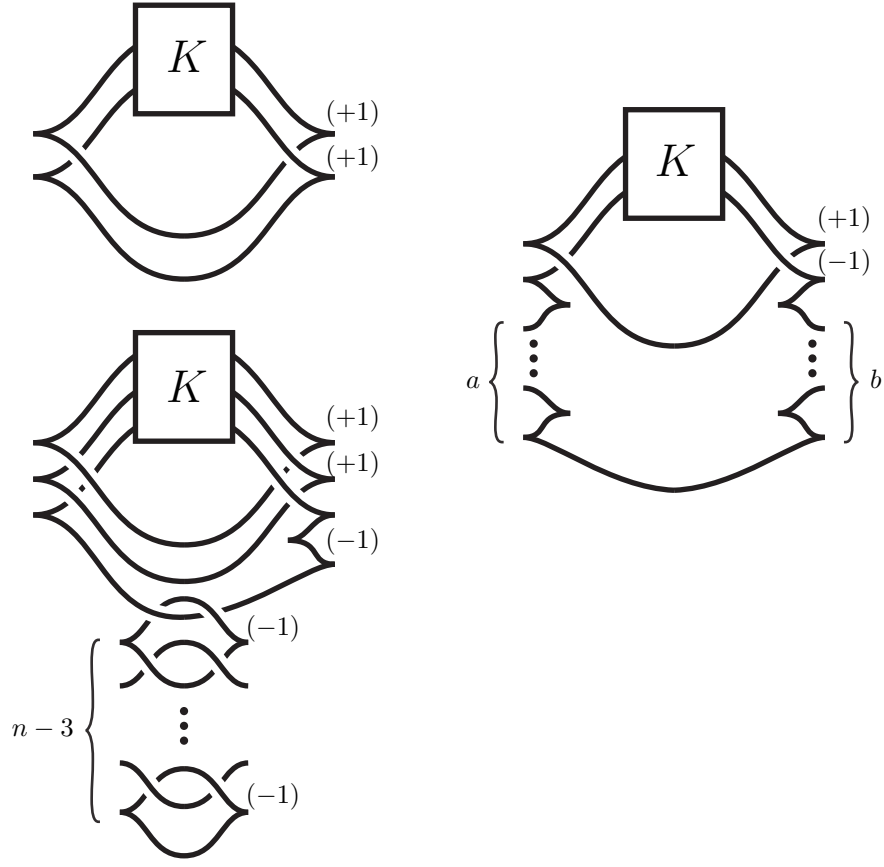
\begin{figure}[htb]{\footnotesize
				\begin{overpic}%[grid,tics=10] 
					{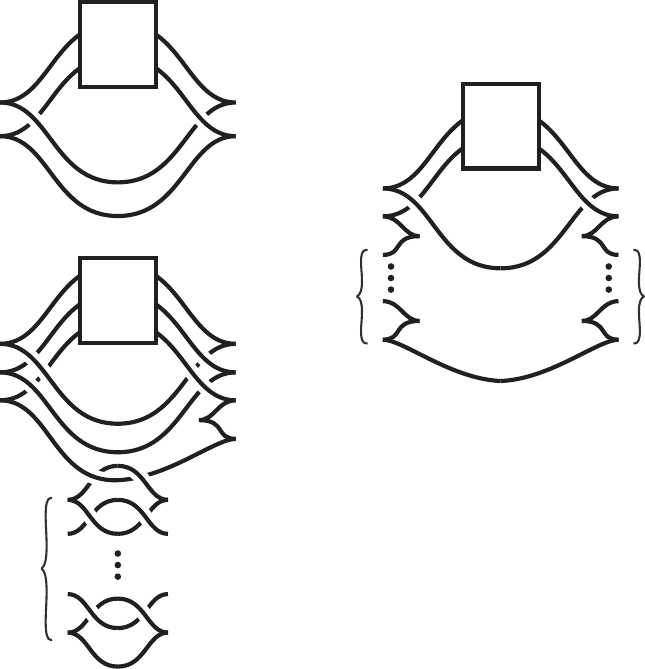}
					\put(50, 296){\Large $K$}
					\put(110, 278){$(+1)$}
					\put(110, 261){$(+1)$}
					\put(50, 172){\Large $K$}
					\put(110, 163){$(+1)$}
					\put(110, 146){$(+1)$}
					\put(110, 117){$(-1)$}
					\put(80, 84){$(-1)$}
					\put(80, 21){$(-1)$}
					\put(-8, 46){$n-3$}
					\put(234, 257){\Large $K$}
					\put(295, 238){$(+1)$}
					\put(295, 222){$(-1)$}
					\put(163, 177){$a$}
					\put(315, 177){$b$}
			\end{overpic}}
			\caption{For a Legendrian knot $K$ with $\tb=-1$ we see a smooth $-1/2$ surgery (that is contact $(1/2)$ surgery) and $-1/n$ surgery for $n>2$ (that is contact $((n-1)/n)$ surgery) on the upper and lower left, respectively, and a smooth $1/n$ surgery (that is a contact $((n+1)/n)$ surgery) on the right. On the right $a$ and $b$ are non-negative integers so that $a+b=n$ (that is the second knot is the Legendrian push-off of the first with $n$ stablizations of one sign or the other). This diagram is taken from Version~1 of \cite{ES} (arXiv).}
			\label{tbm1-general}
		\end{figure}
		
		We begin with $n=2$ and let $X_{\pm 2}$ be the $4$-manifolds given by the surgery diagram in the figure for $\pm 1/2$ surgery. The intersection matrix for $X_{-2}$ is 
		$$
		\begin{bmatrix}
			0 & -1 \\
			-1 & 0 
		\end{bmatrix}
		$$
		One may easily compute that the Euler characteristic is $\chi(X_{-2})=3$ and the signature is $\sigma(X_{-2})=0$. Moreover, since the rotation number of $L$ is zero we see that $c_2^2(X_{-2})=0$. To compute the $d_3$ invariant we note that there were two $+1$ contact surgeries in the diagram so
		\begin{align*}
			d_3(\partial X_{-2})&=\frac 14\left(c_1^2(X_{-2}) -3\sigma(X_{-2})-2(\chi(X_{-2})-1)\right)+2\\
			&=\frac 14 (0-0-4)+2=1.
		\end{align*}
		
		We now consider $X_2$ which has intersection matrix
		$$
		\begin{bmatrix}
			0 & -1 \\
			-1 & -4
		\end{bmatrix}
		$$
		One may easily compute that the Euler characteristic is $\chi(X_{-2})=3$ and the signature is $\sigma(X_{-2})=0$. The bottom Legendrian knot in the surgery diagram has rotation number either $-2,0,$ or $2$. One may check in all cases that $c_1^2(Q)=0$. Noting that there is a single contact $+1$ surgery in the diagram of $X_2$ we see that 
		\[
		d_3(\partial X_2)=\frac 14 (0-0-4)+1=0. 
		\]
		Thus $\pm 1/2$ surgery on a Legendrian with $\tb=-1$ can never give a cosmetic surgery. 
		
		We now turn to the case of $\pm 1/n$ surgery for $n>2$. Denote the $4$-manifold constructed in the surgery diagram for $\pm 1/n$ surgery by $X_{\pm k}$. The intersection matrix of $X_{-n}$ is given by 
		$$
		\begin{bmatrix}
			0 & -1 & -1 & & &  \\
			-1 & 0 & -1 & & &   \\
			-1 & -1 & -3 & -1& &   \\
			&  & -1 & -2 & \ddots & \\
			& &  & \ddots & \ddots  &-1\\
			&   &&&-1 & -2   \\
		\end{bmatrix}_{n \times n}\\
		$$
		One may easily compute that the Euler characteristic is $\chi(X_{-n})=n+1$. To compute the signature, we note that the $4$-manifold $X$ on the left of Figure~\ref{compsig} has the same signature as $X_{-n}$. The manifold $X'$ in the middle of the figure is obtained from $X$ by a $+1$ blowup. Thus $\sigma(X')=\sigma(X)+1$. The manifold $X''$ on the right of the figure is obtained from $X'$ by two $+1$ blowdowns. Thus $\sigma(X'')=\sigma(X')-2$. It is clear that $\sigma(X'')=-n+1$ and thus $\sigma(X_{-n})=\sigma(X)=-n+2$.
		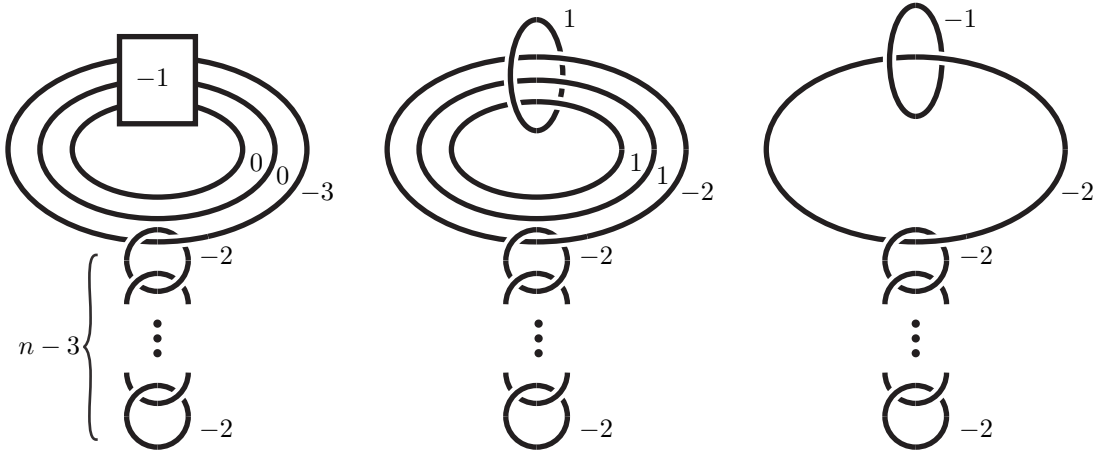
\begin{figure}[htb]{\footnotesize
				\begin{overpic}%[grid,tics=10] 
					{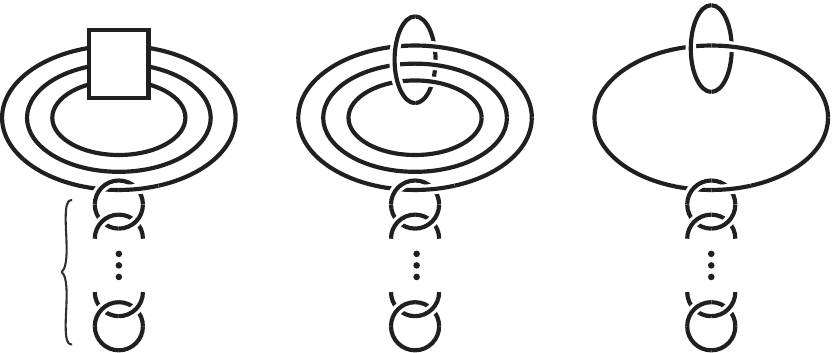}
					\put(49, 137){$-1$}
					\put(92, 105){$0$}
					\put(102, 100){$0$}
					\put(111, 94){$-3$}
					\put(73, 70){$-2$}
					\put(73, 5){$-2$}
					\put(5, 36){$n-3$}
					\put(210, 159){$1$}
					\put(235, 105){$1$}
					\put(245, 100){$1$}
					\put(254, 94){$-2$}
					\put(216, 70){$-2$}
					\put(216, 5){$-2$}
					\put(353, 159){$-1$}
					\put(397, 94){$-2$}
					\put(359, 70){$-2$}
					\put(359, 5){$-2$}
			\end{overpic}}
			\caption{Computing the signature for $X_n$. This diagram is taken from Version~1 of \cite{ES} (arXiv).} 
			\label{compsig}
		\end{figure}
		In Appendix~\ref{lac} we will prove the following lemma. 
		\begin{lemma}\label{computationfortbm1}
			For any choice of stablization in the definition of $X_{-n}$ we have $c_1^2(X_{-n})=2-n$. 
		\end{lemma}
		Now noting that there are two contact $(+1)$ surgeries in the surgery diagram for $X_{-n}$ we see 
		\[
		d_3(\partial X_{-n})=\frac 14((2-n)-3(-n+2)-2n) + 2= 1
		\]
		We now consider $X_n$. To this end, we note the intersection form of $X_n$ is 
		$$
		\begin{bmatrix}
			0 & -1 \\
			-1 & -2-n
		\end{bmatrix}
		$$
		So we easily see that $\chi(X_n)=3$, $\sigma(X_n)=0$, and as above we see that independent of the rotation number of the bottom Legendrian in the surgery diagram that $c_1^2(Q)=0$ and thus
		\[
		d_3(\partial X_n)=\frac 14 (0-0-4)+1=0.
		\]
		Since this does not agree with $d_3(\partial X_{-n})$ we see that there are no cosmetic surgeries on $L$ with surgery coefficients $\pm 1/n$. 
        \end{proof}

		We now turn to Legendrian knots with Thurston-Bennequin invariant $-2$.
		\begin{proposition}\label{p2}
			The contact cosmetic surgery conjecture holds for Legendrian knots with Thurston-Bennequin invariant $-2$. 
		\end{proposition}
		\begin{proof}
			Let $L$ be a Legendrian knot with Thurston-Bennequin invariant $-2$. From Theorems~\ref{tauresult} and~\ref{taubound} we see that $\tb(L)=-2$ and that the rotation number of $L$ must be $\pm 1$. 
			
			Recall we only need to check that smooth $\pm 2$ and $\pm 1/n$ surgeries do not yield cosmetic surgeries. Note that smooth $-2$ surgery on $L$ is a contact $(0)$ surgery and so is not well-defined, so we are left to check $\pm 1/n$ surgeries. We see surgery diagrams for these in Figure~\ref{tbm2-surg}. 
			\begin{figure}[htb]{\footnotesize
					\begin{overpic}%[grid,tics=10] 
						{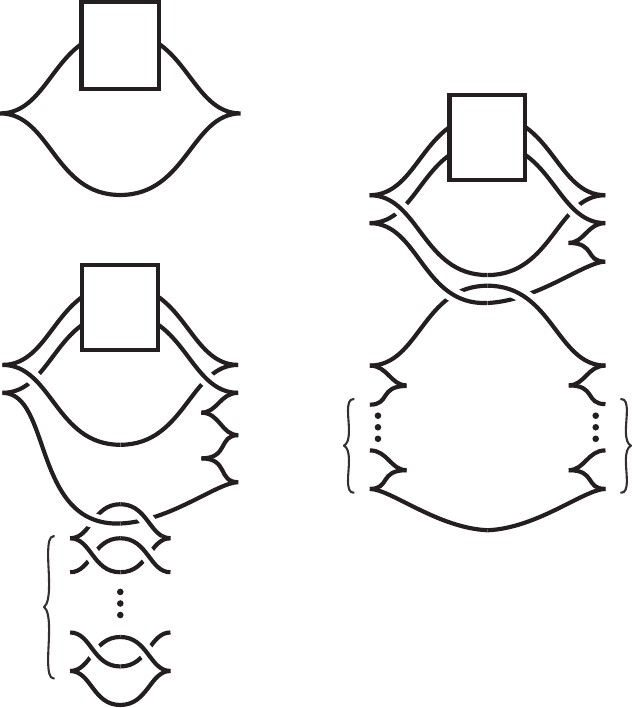}
						\put(51, 313){\Large $K$}
						\put(110, 292){$(+1)$}
						\put(51, 187){\Large $K$}
						\put(110, 170){$(+1)$}
						\put(110, 140){$(-1)$}
						\put(80, 85){$(-1)$}
						\put(80, 21){$(-1)$}
						\put(-8, 46){$n-2$}
						\put(228, 269){\Large $K$}
						\put(291, 251){$(+1)$}
						\put(291, 222){$(-1)$}
						\put(291, 170){$(-1)$}
						\put(158, 124){$a$}
						\put(308, 124){$b$}
				\end{overpic}}
				\caption{For a Legendrian knot $K$ with $\tb=-2$ we see a smooth $-1$ surgery (that is contact $(1)$ surgery) and $-1/n$ surgery for $n>1$ (that is contact $((2n-1)/n)$ surgery) on the upper and lower left, respectively (the stabilizations in the lower diagram can be of any sign), and a smooth $1/n$ surgery (that is a contact $((2n+1)/n)$ surgery) on the right. On the right $a$ and $b$ are non-negative integers so that $a+b=n-1$ (that is the second knot is the Legendrian push-off of the first with $n-1$ stablizations of one sign or the other). This diagram is taken from Version~1 of \cite{ES} (arXiv)}
				\label{tbm2-surg}
			\end{figure}
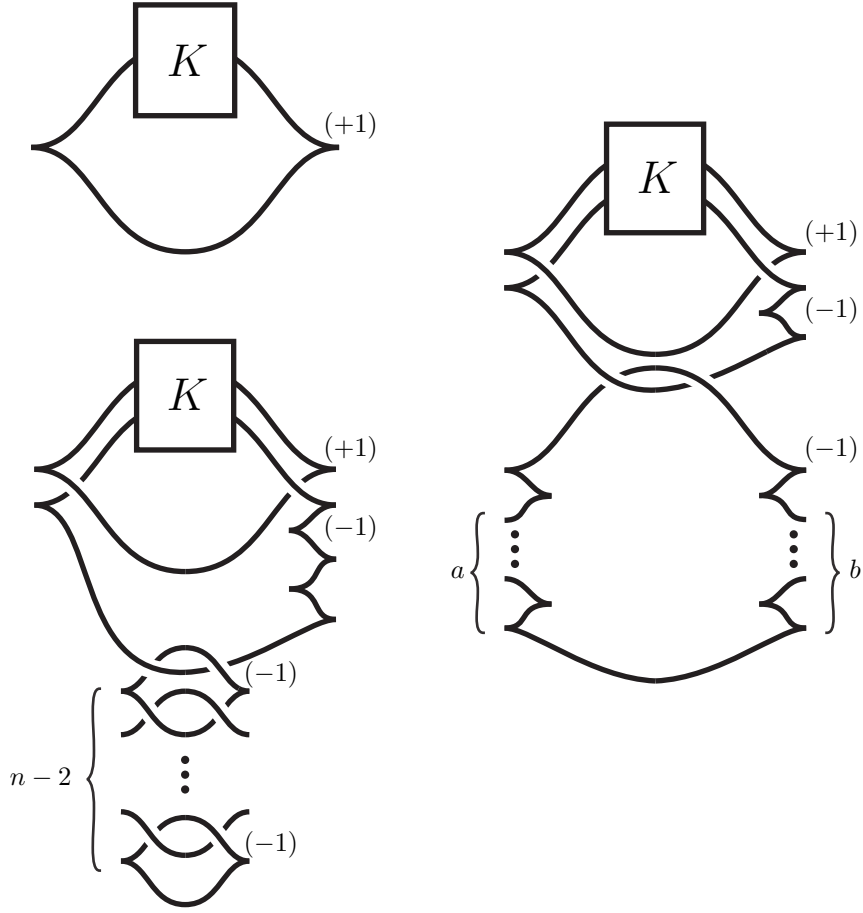
		
			We begin by considering the case when $n=1$. Let $X_{\pm 1}$ denote the manifold obtained from the surgery diagram for $\pm 1$ surgery on $L$. We see that the intersection matrix of $X_{-1}$ is 
			$\begin{bmatrix}
				-1 
			\end{bmatrix}
			$ and hence 
			$\chi(X_{-1})=2$ and  
			$\sigma(X_{-1})=-1$. Whether or not the rotation number of $L$ is $-1$ or $1$ we see $c_1^2(X_{-1})=-1$. Given that there is one contact $+1$ surgery in the diagram we can compute
			\begin{align*}
				d_3(\partial X_{-1})&=\frac 14\left(c_1^2(X_{-2}) -3\sigma(X_{-2})-2(\chi(X_{-2})-1)\right)+1\\
				&=\frac 14 (-1-3(-1)-2(1))+1=1.
			\end{align*}
			
			Moving to $X_1$ we note the intersection matrices of all the $X_n$ are similar 
			$$
			\begin{bmatrix}
				-1 & -2 &0\\
				-2 & -4 &-1\\
				0 & -1 &-n-1
			\end{bmatrix}
			$$
			It is clear that $\chi(X_{1})=4$. One may also compute that $\sigma(X_1)=-1$ and the inverse of the matrix is
			$$
			\begin{bmatrix}
				4n+3 & -2(n+1) &2\\
				-2(n+1) & n+1 &-1\\
				2 & -1 &0
			\end{bmatrix}
			$$
			Using this and the fact that the rotation number of the top Legendrian knot in the diagram is $\pm 1$, of the middle Legendrian knot is $\pm 2$ or $0$ (and the sign on the $2$ must match that of the $1$), and of the bottom Legendrian knot is $0$, we can compute that $c_1^2(X_1)$ is $7$ or $-1$ and hence $d_3= 2$ or $0$. In all these cases we see that this does not agree with $d_3(X_{-1})$ and so $\pm 1$ surgery cannot be a contact cosmetic surgery slope. 
			
			We now turn to the case of $\pm 1/n$ surgery for $n\geq 2$ and let $X_{\pm n}$ be the $4$-manifold given in the figure for $\pm 1/n$ surgery on $L$. The intersection form of $X_{-n}$ is given by 
			$$
			\begin{bmatrix}
				-1 & -2 & & &  \\
				-2 & -5 & -1 & &   \\
				& -1 & -2 & \ddots&    \\
				&  & \ddots & \ddots  &-1\\
				&&&-1 & -2   \\
			\end{bmatrix}_{n \times n}
			$$

			It is clear that $\chi(X_{-n})=n+1$ and in Appendix~\ref{lac} we will prove the following lemma.
			\begin{lemma}\label{tbm2computation}
				The intersection from of $X_{-n}$ is negative definite, so $\sigma(X_{-n})=-n$ and $c_1^2(X_{-n})=-9n+8$ or $-n$. 
			\end{lemma}
			Thus we can compute
			\[
			d_3(X_{-n})=\frac 14(-n+3n-2n)+1=1
			\]
			or 
			\[
			d_3(X_{-n})=\frac 14((-9n+8)+3n-2n)+1=-2n+3
			\]
			
			We now consider $X_n$. The intersection matrix of $X_n$ is given above, from which we compute that $\chi(X_n)=4$ and $\sigma(X_n)=-1$. 
			The possible rotation vectors are
			\[
			\mathbf{r}=
			\begin{bmatrix}
				\pm 1\\
				\pm 2\\
				i
			\end{bmatrix}
			\quad \text{or}\quad
			\begin{bmatrix}
				\pm 1\\ 
				0\\
				i
			\end{bmatrix}
			\]
			where $i=n-1, n-3, \ldots, -n+1$. We note that in the first vector, the signs must be the same. This leads to $c_1^2(X_n)= -1$ or $3+4n\pm 4i$.
			Thus we can compute 
			\[
			d_3(X_n)= \frac 14(-1+3-6)+1=0
			\]
			when $\mathbf{r}=
			\begin{bmatrix}
				\pm 1,
				\pm 2, 
				i
			\end{bmatrix}^T$ and
			\[
			d_3(X_n)=\frac 14(3+4n\pm 4i+3-6))+1=n\pm i+1
			\]
			when $\mathbf{r}=
			\begin{bmatrix}
				\pm 1,
				0, 
				i
			\end{bmatrix}^T$.
		
			We note that these invariants are always even (recall the restriction on $i$), thus they do not agree with the $d_3$-invariants for $X_{-n}$ and there are no $\pm 1/n$ cosmetic surgeries on a Legendrian knot with $\tb=-2$.
		\end{proof}
		
		We finally consider Legendrian knots with $\tb<-2$. 
		\begin{proposition}\label{p3}
			The contact cosmetic surgery conjecture holds for Legendrian knots with Thurston-Bennequin invariant less than $-2$. 
		\end{proposition}
		\begin{proof}
      
			Recall we only need to consider the smooth surgeries $\pm 2$ and $\pm 1/n$ by Theorem~\ref{adaptedHanselman}.

			Since $\pm 2$ case is identical to \cite{ES}, we now turn to $\pm 1/n$ surgeries on $L$. They are indicated in Figure~\ref{1unsurgeriesgeneral}.  It will be convenient to consider $\pm 1$ surgeries first and denote the $4$-manifolds given in the figure by $X_{\pm 1}$. 
			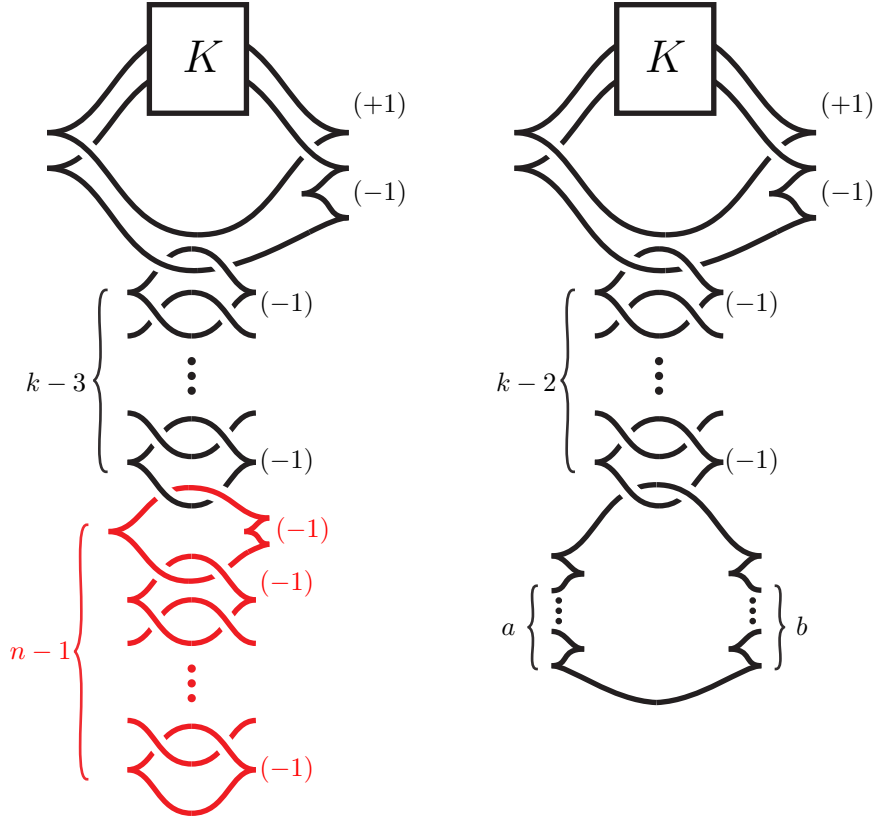
\begin{figure}[htb]{\footnotesize
					\begin{overpic}
						{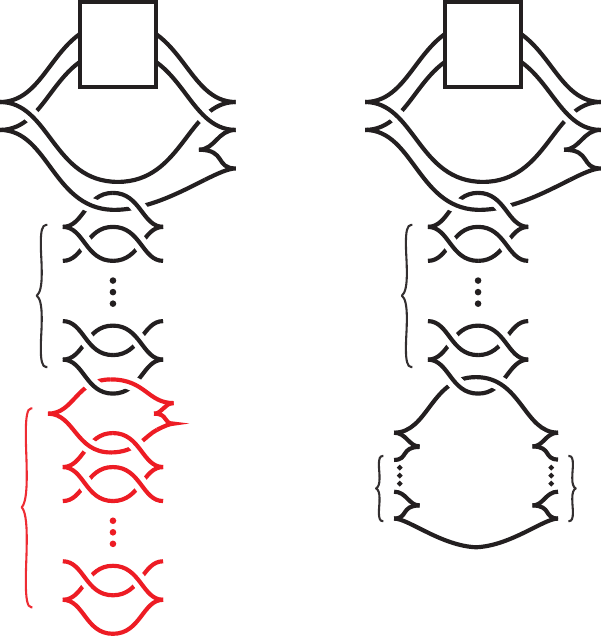}
						\put(51, 280){\Large $K$}
						\put(115, 265){$(+1)$}
						\put(115, 232){$(-1)$}
						\put(80, 190){$(-1)$}
						\put(80, 131){$(-1)$}
						\put(80, 85){\color{red}$(-1)$}
						\put(80, 15){\color{red}$(-1)$}
						\put(86, 105){\color{red}$(-1)$}
						\put(-8, 160){$k-3$}
						\put(-14, 59){\color{red}$n-1$}
						\put(225, 280){\Large $K$}
						\put(291, 265){$(+1)$}
						\put(291, 232){$(-1)$}
						\put(255, 190){$(-1)$}
						\put(255, 131){$(-1)$}
						\put(169, 160){$k-2$}
						\put(171, 68){$a$}
						\put(282, 68){$b$}
				\end{overpic}}
				\caption{For a Legendrian knot $K$ with $\tb=-k<-2$ we see a smooth $-1/n$ surgery (that is contact $((kn-1)/n)$ surgery) on the left, (the stabilizations in the lower diagram can be of any sign). The red portion of the diagram should be ignored fo $n=1$. On the right, we see smooth $1/n$ surgery (that is a contact $((kn+1)/n)$ surgery) on the right. On the right $a$ and $b$ are non-negative integers so that $a+b=n-1$ (that is the second knot is the Legendrian push-off of the first with $n-1$ stablizations of one sign or the other). This diagram is taken from Version~1 of \cite{ES} (arXiv).}
				\label{1unsurgeriesgeneral}
			\end{figure}
			
			We begin with $-1$ surgeries. From the surgery diagram, we see that the intersection matrix is
			$$ 
			\begin{bmatrix}
				-k+1 & -k &  & &   \\
				-k & -k-2 & -1 & &   \\
				& -1 & -2 & \ddots&    \\
				&  & \ddots &\ddots & -1  \\
				&   &&-1 & -2   \\
			\end{bmatrix}_{l \times l}
			$$
			where $l=k-1$. From this we can easily determine that $\chi(X_{-1})=k$ and in Appendix~\ref{lac} we will prove the following lemma.
			\begin{lemma}\label{lem1}
				The signature of $X_{-1}$ is $\sigma(X_{-1})=-k+1$ and 
				\[
				c_1^2(X_{-1})=-i^2\pm 2(k-2)i -k^2+3k-2,
				\]
				where $i=k-1,k-3,\ldots, -k+1$. 
			\end{lemma}
			Thus we can compute
			\[
			d_3(\partial X_{-1})=\frac 14\left( -i^2\pm 2(k-2)i-k^2+4k-3\right)+1.
			\]
			
			Turning to $+1$ surgery we see the intersection matrix of $X_1$ is given by the above matrix with $l=k+1$. From this we see that $\chi(X_1)=k+2$ and in Appendix~\ref{lac} we will prove the following lemma. 
			\begin{lemma}\label{lem2}
				The signature of $X_{1}$ is $\sigma(X_{1})=-k+1$ and 
				\[
				c_1^2(X_{1})=i^2\mp 2ki+k^2-k,
				\]
				where $i=k-1,k-3,\ldots, -k+1$. 
			\end{lemma}
			Thus we can compute
			\[
			d_3(\partial X_1)=\frac 14\left(i^2\mp 2ki +k^2-5\right)+1.
			\]
			Using the quadratic equation (or Mathematica) to solve $d_3(\partial X_{-1})=d_3(\partial X_1)$ for $i$ yields no integer solutons. So there are no contact cosmetic surgeries with smooth surgery coefficients $\pm 1$ when $\tb(L)<-2$. 
			
			We finally turn to $\pm 1/n$ surgery when $n>1$. Let $L$ be a Legendrian knot with $\tb(L)=-k<-2$. We let $X_{\pm n}$ denote the $4$-manifold described in Figure~\ref{1unsurgeriesgeneral} with boundary $\pm 1/n$ surgery on $L$.
			
			The intersection matrix of $X_{-n}$ is given by 
			$$
			\begin{bmatrix}
				-k+1 & -k &  & &  & \\
				-k & -k-2 & -1 & &   &\\
				& -1 & -2 & -1&   & \\
				&  &  \ddots &\ddots &  \ddots&\\
				&&  &  -1 &-2 &  -1&\\
				&&&  &  -1 &-3 &  -1&\\
				&&&&  &  -1 &-2 &  -1&\\
				&&&&	&  &  \ddots &\ddots &  \ddots&\\
				&&&&&   &&-1 & -2& -1  \\
				&&&&&   &&&-1 & -2  \\
			\end{bmatrix}_{k+n-2 \times k+n-2}
			$$ 
			where the $-3$ occurs in the $(k,k)$ entry. 
			So it is clear that $\chi(X_{-n})=k+n-1$. From Theorem~\ref{tauresult} and~\ref{taubound} we know that the rotation number of a Legendrian with $tb=-k<-2$ is $i=k-1,k-3,\ldots, -k+1$. So the rotation vector for the surgery diagram is $\mathbf{r}=\begin{bmatrix}i,i\pm 1, 0, \cdots 0,  j, 0 \cdots, 0\end{bmatrix}^T$ where the second $j$ is in the $k^{th}$ entry and takes the values $\pm 1$. In Appendix~\ref{lac} we will prove the following lemma.

			\begin{lemma}\label{lem10}
				The signature of $X_{-n}$ is $\sigma(X_{-n})=-k-n+2$ and 
			
				$$c_1^2(X_{-n})=-n +1+(1-k)(-1+(k-1)n)-j2(-1)^k(n-1)(-i\pm k \mp 1)\pm 2i(-1+n(k-1))-ni^2.$$
			
			\end{lemma}
			Thus we can compute, for $j=1$ 
			\begin{align*}
				d_3(\partial X_{-n})=&\frac 14 \left(-n +1+(1-k)(-1+(k-1)n)-j2(-1)^k(n-1)(-i\pm k \mp 1)+\right.\\ 
				&  \left. \pm 2  i ((k-1) n-1)-n i^2+ k+n-2 \right) +1 
			\end{align*}

			The intersection matrix of $X_n$ is given by 
			$$
			\begin{bmatrix}
				-k+1 & -k &  & &  & \\
				-k & -k-2 & -1 & &   &\\
				& -1 & -2 & \ddots&   & \\
				&  &  \ddots &\ddots &  \ddots&\\
				&   &&\ddots & -2& -1  \\
				&   &&&-1 & -n-1  \\
			\end{bmatrix}_{k+1 \times k+1}
			$$
			So it is clear that $\chi(X_n)=k+2$. Notice that the rotation vector for $X_{n}$ is $$\mathbf{r}=\begin{bmatrix}i,i\pm 1, 0, \cdots, 0,s\end{bmatrix}^T.$$  
			In Appendix~\ref{lac} we will prove the following lemma.
			
			\begin{lemma}\label{lem11}
				The signature of $X_n$ is $\sigma(X_n)=-k+1$ and
			
				$$c_1^2(X_{n})=-1+k-2 k n \left(2 i^2\pm 3 i+1\right)+nk^2 (1\pm 2i)^2+2s (-1)^k (i\mp k\pm 1) +n (i\pm 1)^2\mp2 i.$$
			\end{lemma}
			Thus we can compute $d_3$

			\begin{align*}
				d_3(\partial X_n)=&\frac 14 \left(-1+k-2 k n \left(2 i^2\pm 3 i+1\right)+nk^2 (1\pm 2i)^2+ 2s(-1)^k(i\mp k\pm 1)+\right.\\ 
				&  \left.n (i\pm 1)^2\mp2 i+ k+5 \right) +1 
			\end{align*}

			We use Mathematica to solve $d_3(\partial X_{-n})=d_3(\partial X_n)$. For ease, we split the computation into cases based on the parity of $k$, the sign of $j$, and the sign of the stabilizations of the push-off of the Legendrian knot (these are the $\pm$ in the formulas). Consider the case when $k$ is even, $j=1$, and the stabilizations are positive. Solving for $n$ we get 
			$$n=\frac{ks+k-i s-i-s+1}{2k^2 i^2+2 k^2 i+k^2-2ki^2-4ki-k+i^2+i}.$$ 
			We then impose the condition that $-n<s<n$ (recall that $s$ is the rotation number of a Legendrian unknot with $\tb=-n$ in the surgery diagram and for Legendrian unknots $\tb<\rot<-\tb$). There are no integral solutions in our range. For the other cases, a similar analysis also shows there are no solutions.
			
			So there are no contact cosmetic surgeries with smooth surgery coefficients $\pm \frac{1}{n}$ when $\tb(L)<-2$. 
		\end{proof}
	
		\appendix
		\section{Linear algebra computations}\label{lac}
		
		We recall a few facts about matrices. Given an invertible $n\times n$ matrix $A$ with entries $a_{i,j}$, denote its inverse by $B$ with entries $b_{i,j}$. We can compute the entry $b_{i,j}$ as follows
		\[
		b_{i,j}=(-1)^{i+j}\frac{\det A_{j,i}}{\det A}
		\]
		where $A_{i,j}$ is the $(i,j)$ minor of $A$, that is the $(n-1)\times (n-1)$ matrix obtained from $A$ by deleting the $i^{th}$ row and $j^{th}$ column. 
		
		The following matrix will appear frequently in our discussion
		\[
		I_n=
		\begin{bmatrix}
			-2 & -1 & &   \\
			-1 & -2 & \ddots&    \\
			
			& \ddots & \ddots  &-1\\
			&&-1 & -2   \\
		\end{bmatrix}_{n \times n}
		\]
		as will the matrix
		\[
		I'_n=
		\begin{bmatrix}
			-1&-1&0&\cdots &0\\
			0& -2 & -1 & &   \\
			0&-1 & -2 & \ddots&    \\
			
			\vdots&  & \ddots & \ddots  &-1\\
			0& &&-1 & -2   \\
		\end{bmatrix}_{n \times n}
		\]
		One may easily compute that $\det I'_n=-\det I_{n-1}$.
		\begin{lemma}
			The determinant of $I_n$ is 
			\[
			\det I_n= (-1)^n(n+1).
			\]
		\end{lemma}
		\begin{proof}
			It is clear that $\det I_1=-2$ and $\det I_2=3$. Using the formula
			\[
			\det I_n=-2\det I_{n-1} -(-1)\det I'_{n-1},
			\]
			one may easily establish the formula via induction. 
		\end{proof}
		
		We will frequently use the following lemma.
		\begin{lemma}\label{negdefproof}
			Given a matrix of the form
			\[
			M=\begin{bmatrix}
				A & B\\
				B^T & I_n
			\end{bmatrix}
			\]
			where $I_n$ is the matrix above
			\[
			A=\begin{bmatrix}
				a&b\\
				b&c
			\end{bmatrix}
			\]
			and
			\[
			B=\begin{bmatrix}
				0&0&\ldots &0\\
				-1&0&\ldots &0
			\end{bmatrix}
			\]
			is a $2\times n$ matrix. Then 
			\[
			\det M = (-1)^n((a(c+1)-b^2)n+(ac-b^2)).
			\]
			Thus if $ac-b^2$ is positive and $a(c+1)-b^2$ is non-negative, then $M$ is negative definite.
		\end{lemma}
		\begin{proof}
			If $D$ is an invertible matrix then the determinant of a matrix of the form 
			$$
			\begin{bmatrix}
				A & B  \\
				C & D  \\
			\end{bmatrix}
			$$
			is $\det(D)\det(A-BD^{-1}C),$ where $A$ and $D$ are square matrices of size $i$ and $n-i$ and $B$ is ($i\times n-i$) matrix and $C$ is ($n-i \times i$) matrix. From this we see that 
			\[
			\det M= \det I_n \det (A-BI^{-1}_nB^T).
			\]
			One may easily see that $BI_nB^T$ is the matrix $\begin{bmatrix}
				0&0\\ 0& a_{1,1}
			\end{bmatrix}$
			where $a_{1,1}$ is the upper-left entry in $I_n^{-1}$. From the formula recalled at the beginning of this appendix we see that $a_{1,1}=\det I_{n-1}/\det I_{n}=-n/(n+1)$. Thus
			\[
			\det M= (-1)^n (n+1) [(a(c+n/(n+1))-b^2]=(-1)^n((a(c+1)-b^2)n+(ac-b^2)n)
			\]
			as claimed.
			
			For the second claim, we recall Sylvester's criterion \cite[Theorem~7.2.5]{HJ} says that a matrix $M$ is negative definite if and only if $(-1)^{k}\det M_k>0$ where $M_k$ is the $k^{th}$ principle minor, that is it is the $k\times k$ submatrix in the upper left corner of $A$. 
		\end{proof}

		We now move to the proof of Lemma~\ref{computationfortbm1}. Recall this lemma claims that the intersection matrix of $X_{-n}$ shown on Page~\pageref{computationfortbm1} has signature $-n+2$ and $c_1^2(X_{-n})=2-n$. 
		\begin{proof}[Proof of Lemma~\ref{computationfortbm1}]\label{opev}

			We begin by showing that $\det Q=(-1)^{n-1}$. We use the formula for the determinate of a block matrix as in Lemma \ref{negdefproof} where 
			\[ A=
			\begin{bmatrix}
				0 & -1 & -1 \\
				-1 & 0 & -1   \\
				-1 & -1 & -3   \\
			\end{bmatrix} \\
			\] 
			$D= I_{n-3}$ and $B$ is as in the lemma. 
			
			We note the rotation vector for $X_{-n}$ is $\mathbf{r}=\begin{bmatrix}0,0,\pm 1, 0, \cdots ,0\end{bmatrix}^T$. For any $n\times n$ matrix $A$ it is clear that $\mathbf{r}^TA\mathbf{r}$ is simply $(\det A)^{-1}a_{3,3}$ where $a_{i,j}$ is the $(i,j)$ entry of $A$. If we denote the intersection matrix of $X_{-n}$ by $Q$ then $c_1^2(X_{-n})$ is the $(3,3)$ entry of $Q^{-1}$ and from the formula above this is simply $(-1)^{n-1}\det Q_{3,3}$. Since $Q_{3,3}$ is a block diagonal matrix we easily see that $\det Q_{3,3}=(-1)\det I_{n-3}=(-1)^n(n-2)$ and thus $c_1^2(X_{-n})=2-n$. 
		\end{proof}
		
		We now turn to the proof of Lemma~\ref{tbm2computation} which says the intersection matrix of $X_{-n}$ (for $\tb=-2$ Legendrian knots) shown on Page~\pageref{tbm2computation} is negative definite and $c_1^2(X_{-n})=-9n+8$ or $-n$.
		\begin{proof}[Proof of Lemma~\ref{tbm2computation}]
			The fact that the intersection matrix fo $X_{-n}$ is negative definite follows directly from Lemma~\ref{negdefproof}. 
			
			The rotation vector for $X_{-n}$ is $\mathbf{r}=\begin{bmatrix}i,j, 0,\ldots, ,0\end{bmatrix}^T$ where $i=\pm 1$ and if $i=1$ then $j$ could be $3, 1,$ or $-1$ while if $i=-1$ then $j$ could be $-3, -1,$ or $1$. For any symmmetric matrix $A$ the quantity $\mathbf{r}^TA\mathbf{r}=i^2a_{1,1}+ 2ija_{1,2} + j^2a_{2,2}$. Thus if $Q$ is the intersection matrix for $X_{-n}$ and $Q'$ is its inverse then $c_1^2(X_{-n})= i^2q'_{1,1}+ 2ijq'_{1,2} + j^2q'_{2,2}$. Now we compute
			\begin{align*}
				q'_{1,1} &= (-1)^{2}(det Q)^{-1} \begin{vmatrix}
					-5 & -1 & &   \\
					-1 & -2 & \ddots&    \\
					& \ddots & \ddots  &-1\\
					&&-1 & -2   \\
				\end{vmatrix}_{n-1 \times n-1}\\%%%
				&=(-1)^n \left( -5 \det I_{n-2}
				+\det I'_{n-2}\right)\\%%%
				&=(-1)^n((-1)^{n-1}5(n-1)-(-1)^{n-1}(n-2))\\
				&=-(5n-5-n+2)=3-4n
			\end{align*}
			and
			\begin{align*}
				q'_{2,1}&= (-1)^3(det Q)^{-1}  \begin{vmatrix}
					-2  & -1 & &   \\
					& -2 & \ddots&    \\
					& \ddots & \ddots  &-1\\
					&&-1 & -2   \\
				\end{vmatrix}\\%%%
				&=(-1)^{n+1}(-2\cdot \det I_{n-2})=(-1)^{n+1}(-1)^{n-2}(2-2n)\\
				&=2n-2
			\end{align*}
			and
			\begin{align*}
				q'_{2,2}&= (-1)^4(det Q)^{-1}  \begin{vmatrix}
					-1 &  & & &  \\
					& -2 & -1 & &   \\
					& -1 & -2 & \ddots&    \\
					&  & \ddots & \ddots  &-1\\
					&&&-1 & -2   \\
				\end{vmatrix}\\%%%
				&= (-1)^n(-1\cdot \det I_{n-2})=1-n.
			\end{align*}
			From this we compute that $c_1^2(X_{-n})=-9n+8$ or $-n$. 
		\end{proof}

			$$
			\begin{bmatrix}
				-k+1 & -k &  & &   \\
				-k & -k-2 & -1 & &   \\
				& -1 & -2 & \ddots&    \\
				&  & \ddots &\ddots & -1  \\
				&   &&-1 & -2   \\
			\end{bmatrix}_{l+1 \times l+1} or 
			\begin{bmatrix}
				-k-2 & -1 & &   \\
				-1 & -2 & \ddots&    \\
				& \ddots &\ddots & -1  \\
				&&-1 & -2   \\
			\end{bmatrix}_{l+1 \times l+1}
			$$
		%	We may use Lemma~\ref{negdefproof} to show that the first matrix has determinant $(-1)^{l-1}(k^2+(2k-l)-2)$ which clearly has sign $(-1)^{l+1}$ (recall that $l$ is less than $k$). The second matrix is easily seen to be negative definite, so it must have determinant $(-1)^{l+1}$ too by Sylvester's criterion mentioned above. This completes the induction for $l$ up to $k$. When $l$ is $k$, then the inductive step is the same except for the two non-block matrices of size $(k+1)\times (k+1)$. These matrices are as in the previous equations except $l+1=k+1$ in this case. One may easily compute the determinant of the first matrix is $(-1)^k$ which has the wrong sign, but the determinant of the second matrix is $(-1)^{k+1}(k^2+2k+2)$. So their sum has sign $(-1)^{k+1}$ as desired. Finally, one may compute that $E_{k+2}(Q)=\det Q=(-1)^{k+1}$.
			
			%The computation of $c_1^2(X_2)$ is identical to the proof of Lemma~\ref{tbleqm3andm2surgery} except for the $q'_{i,j}$ we have
%			\[
		%	q'_{1,1}=\frac 12(k^2+2k+2), \quad q'_{1,2}=-\frac 12(k^2+k),  \text{and } \quad q'_{2,2}=\frac 12 (k^2-1).
			%\]
		%\end{proof}
		
		Turning to the proof of Lemma~\ref{lem1} we recall that the lemma claims that 
		the signature of $X_{-1}$ is $\sigma(X_{-1})=-k+1$ and 
		$
		c_1^2(X_{-1})=-i^2\pm 2(k-2)i -k^2+3k-2,
		$
		where $i=k-1,k-3,\ldots, -k+1$. 
		\begin{proof}[Proofs of Lemmas~\ref{lem1}]
			The fact that the intersection matrix of $X_{-1}$ is negative definite follows from Lemma~\ref{negdefproof}. The computation of $c_1^2$ is identical to the proof of Lemma 4.4 in \cite{ES} except in Lemma~\ref{lem1} the $q'_{i,j}$ are
			\[
			q'_{1,1}=-k^2+k+1, \quad q'_{1,2}= k^2-2k,  \text{and } \quad q'_{2,2}= -k^2+3k-2,
			\]
		\end{proof}
		
		We next prove Lemma~\ref{lem2} that says the signature of $X_{1}$ is $\sigma(X_{1})=-k+1$ and 
		$c_1^2(X_{1})=i^2\mp 2ki+k^2-k,$
		where $i=k-1,k-3,\ldots, -k+1$. 
		
		\begin{proof}[Proofs of Lemmas~\ref{lem2}]
			The intersection matrix of $X_{1}$  has only one positive eigenvalue.  The proof for this is very similar to the one given in the proof of Lemma 4.4 in \cite{ES} and is left to the reader. The computation of $c_1^2$ is identical to the proof of Lemma 4.4 in \cite{ES} except in Lemma~\ref{lem10} the 
			$q'_{i,j}$ are
			\[
			q'_{1,1}=k^2+k+1, \quad q'_{1,2}= -k^2,  \text{and } \quad q'_{2,2}= k^2-k.
			\]
		\end{proof}
		
		Recall Lemma~\ref{lem10} says that the signature of $X_{-n}$ is $\sigma(X_{-n})=-k-n+2$ and 
		$$c_1^2(X_{-n})=-n +1+(1-k)(-1+(k-1)n)-j2(-1)^k(n-1)(-i\pm k \mp 1)\pm 2i(-1+n(k-1))-ni^2$$
		\begin{proof}[Proofs of Lemmas~\ref{lem10}]
			The intersection matrix of $X_{-n}$ is negative definite by Lemma~\ref{negdefproof}. The computation of $c_1^2$ is as follows: 
			$$i^2 q'_{1,1} +2i(i\pm1)q'_{1,2}\pm 2iq'_{1,k}+ (i\pm 1)^2q'_{2,2}\pm 2(i\pm 1)q'_{2,k}+q'_{k,k}.$$

			The required $q'_{i,j}$ are:
			\[
			q'_{1,1}=-k^2n+k+1, \quad q'_{1,2}=(k)(kn-1-n),  \quad q'_{2,2}=(-k+1)(kn-1-n), \] \[  q'_{1,k}=(-1)^k(k)(-1+n),  \quad q'_{2,k}=(-1)^k(-k+1)(-1+n) \quad \text{and } \quad q'_{k, k}=-n+1.\]
			$\mathbf{r}=\begin{bmatrix}i,i\pm 1, 0, \cdots 0,  j, 0 \cdots, 0\end{bmatrix}^T$. This gives
			$$c_1^2(X_{-n})=-n +1+(1-k)(-1+(k-1)n)-j2(-1)^k(n-1)(-i\pm k \mp 1)\pm 2i(-1+n(k-1))-ni^2$$
			as claimed.
		
		\end{proof}
		
		We finally turn to the proof of Lemma~\ref{lem11} which claims that the signature of $X_n$ is $\sigma(X_n)=-k+1$ and
	
		$$c_1^2(X_{n})=-1+k-2 k n \left(2 i^2\pm 3 i+1\right)+nk^2 (1\pm 2i)^2+2s (-1)^k (i\mp k\pm 1) +n (i\pm 1)^2\mp2 i.$$
		\begin{proof}[Proof of Lemmas~\ref{lem11}]

			The intersection matrix of $X_{n}$  has only one positive eigenvalue. The proof for this is very similar to the one given in the proof of Lemma 4.4 in \cite{ES} and is left to the reader.
			
			The computation of $c_1^2$ is as follows: 
			$$i^2 q'_{1,1} +2i(i\pm 1)q'_{1,2}+2isq'_{1,k+1}+ (i\pm 1)^2q'_{2,2}+ 2(i \pm 1)sq'_{2,k+1}.$$

			The required $q'_{i,j}$ are:		\[
			q'_{1,1}=k^2n+k+1, \quad q'_{1,2}=-k^2n+kn-k,  \quad q'_{2,2}=(k-1)^2n+k-1, \] \[  q'_{1,k+1}=(-1)^k(k),  \quad q'_{2,k+1}=(-1)^k(-k+1) \quad \text{and } \quad q'_{k+1, k+1}=0.\]
			
			We note there is one difference in the computation. The rotation vector for $X_{n}$ is $\mathbf{r}=\begin{bmatrix}i,i\pm 1, 0, \cdots, 0,s\end{bmatrix}^T$. But since we can compute that $q'_{k+1,k+1}=0$ the last term will not contribute. This gives  
			$$c_1^2(X_{n})=-1+k-2 k n \left(2 i^2\pm 3 i+1\right)+nk^2 (1\pm 2i)^2+2s (-1)^k (i\mp k\pm 1) +n (i\pm 1)^2\mp2 i$$
		\end{proof}
		
\bibliographystyle{alpha}
\bibliography{references}

\end{document}